\documentclass[12pt, twoside, a4paper]{amsart}

\usepackage{amsthm,amssymb,amsmath,amsfonts}
\usepackage{enumerate}
\usepackage{mathtools}
\usepackage[a4paper, total={6in, 8in}]{geometry}
\usepackage{hyperref}

\renewcommand{\Re}{\operatorname{Re}}
\renewcommand{\Im}{\operatorname{Im}}
\newcommand{\level}[1]{d_\tau(#1)}
\newcommand{\norm}{\Vert}
\DeclareMathOperator{\Hol}{Hol}
\newcommand{\conj}[1]{\overline{#1}}
\newcommand{\inner}[2]{\langle #1, #2 \rangle}
\newcommand{\normD}[1]{\Vert #1 \Vert_{\mathcal{D}}}
\newcommand{\normH}[1]{\Vert #1 \Vert_{\text{H}^2}}

\newcommand{\CC}{\mathbb{C}}

\newcommand{\eit}[1]{e^{i#1}}
\DeclareMathOperator{\capacity}{Cap}
\DeclareMathOperator{\supp}{supp}
\newcommand{\DD}{\mathbb{D}}
\newcommand{\TT}{\mathbb{T}}
\newcommand{\cX}{\mathcal{X}}
\newcommand{\cH}{\mathcal{H}}
\newcommand{\cZ}{\mathcal{Z}}
\newcommand{\NN}{\mathbb{N}}
\newcommand{\cM}{\mathcal{M}}
\newcommand{\cD}{\mathcal{D}}
\newcommand{\cS}{\mathcal{S}}

\newtheorem{theorem}{Theorem}

\newtheorem{thm}{Theorem}
\newtheorem{lem}[thm]{Lemma}
\newtheorem{prop}[thm]{Proposition}
\newtheorem{colo}[thm]{Corollary}
\newtheorem{defin}{Definition}

\newtheorem*{prop*}{Proposition}

\title{Onto interpolation for the Dirichlet space and for  $ H_1(\DD) $}
\author{Nikolaos Chalmoukis} 
\date{}

\thanks{ The author is supported by the fellowship INDAM-DP-COFUND-2015  ``INdAM Doctoral Programme
	in Mathematics and/or Applications
	Cofunded by Marie Sklodowska-Curie Actions", Marie Sklodowska-Curie Actions \# 713485}


\address{N. Chalmoukis \\ Dipartimento di Matematica \\ Universit\'a di Bologna \\ Piazza di Porta S. Donato, 5 \\ 40126 Bologna, Italy}
\email{nikolaos.chalmoukis2@unibo.it}

\begin{document}

\begin{abstract} We give a characterization  of onto interpolating sequences with finite associated measure for the Dirichlet space in terms of condenser capacity. In the Sobolev space $H_1(\DD)$ we define a natural notion of onto interpolation and we prove that the same condenser capacity condition characterizes all onto interpolating sequences.  As a result, for sequences with finite associated measure, the problem of interpolation by an analytic function reduces to a problem of interpolation by a function in $H_1(\DD)$.

\end{abstract}

\maketitle


\section{Introduction}

Interpolation problems by analytic functions is a subject of more than a century old that is rich of interesting and deep results but also continues to stimulate new research. Interpolation has come a long way since the fundamental papers of Pick \cite{Pick1915} and Nevanlinna \cite{Nevanlinna1919}, but the spirit of the problems is invariant. One is given a subset of analytic functions in the unit disc  $\varepsilon  \subseteq \Hol(\DD)$, a sequence (finite of infinite) $\{z_i\}\subseteq \DD$ of {\it interpolating nodes} and a {\it target space} $\cX$, i.e. a set of sequences to be interpolated. The problem is then to determine whether or not for any data $\{ w_i \} \in \cX$ there exists a function $f\in \varepsilon$ such that 
\[ f(z_i)=w_i, \,\,\, \forall i. \]
Before entering into the particulars of the interpolation problem that we will consider let us mention that such problems are ubiquitous in analysis, with applications for example to the theory of Banach algebras (Corona theorem for $H^\infty$ \cite{Carleson1962}, analytic discs in the maximal ideal subspace  \cite{Schark1961} etc),  but also in applications (if one is interested in the latter aspect we recommend the survey paper \cite{McCarthy2003}).

Let us now take a brief look at the interpolation theory for the Hardy space $H^2$ and $H^\infty$, the algebra of bounded analytic functions which as we shall see are closely related. 

\subsection{The Hardy space case}

Consider the Hardy space $H^2$, the Hilbert space of analytic functions in the unit disc equipped with the norm 
 \[ \norm f \norm ^2_{H^2} := \sup_{0<r<1} \frac{1}{2\pi} \int_{0}^{2\pi} |f(re^{i\theta} )|^2 d\theta.\] 
 
 The Szeg\"o kernel,
 \[ \cS_w(z):=\frac{1}{1-z\conj{w}}, \]
 plays the role of the {\it reproducing kernel} for the Hardy space, meaning that for any $f\in H^2$
 \[ f(w) = \inner{f}{\cS_w}_{H^2}, \,\,\,  \forall w\in \DD. \]
The (commutative) Banach algebra $ H^\infty $ is in fact closely related to $H^2$ since it is isometric and weak* homeomorphic to the algebra of multiplication operators on $H^2$ \cite[Theorem 3.24]{AglerMcCarthy00}. By a multiplication operator we mean a bounded linear operator $M_g$ where $g$ is the symbol of the operator, defined by the relation 
\[ M_gf:=gf, \,\,\, \forall f \in H^2. \]

Going back to interpolation, for $H^\infty$ the interpolation problem seems quite natural. Given a sequence of nodes $\cZ=\{z_i\}$  is it always possible to solve the interpolation problem for bounded data $\{w_i\}\in \ell^\infty$ by a function in $H^\infty$ ? This problem has already been considered by Carleson \cite{Carleson1958}.

About the corresponding interpolation problem for $H^2$, the target space seems less obvious.  Following Shapiro and Shields \cite{Shapiro61} we define the target space to be sequences $\{ a_i \}$ of complex numbers such that 
\[ \sum_{i=1}^\infty|a_i|^2(1-|z_i|^2) < +\infty. \]
Therefore Shapiro and Shields ask for which sequences the weighted restriction operator 
\[ T(f)= \{f(z_i)(1-|z_i|^2)^{\frac{1}{2}} \} \]
satisfies 
\begin{align}
&T(H^2) = \ell^2 \tag{$UI_{H^2}$} \label{UIH2}\\ 
&T(H^2) \supseteq \ell^2.  \tag{$OI_{H^2}$} \label{OIH2}
\end{align}
If $\eqref{UIH2}$ happens we say that the sequence is {\it universally interpolating}, and if \eqref{OIH2} happens we call it {\it onto interpolating}. Notice that a difficulty that arises is already reflected in the terminology we have choosen. {\it A priori} it is not clear if one should ask for the  restriction operator to be bounded or just that it is possible to solve the interpolation problem for any given data in the target space. Fortunately, for $H^2$ this ceases to be a concern, since both notions turn out to be equivalent. This is part of the results of Carleson \cite{Carleson1958} and Shapiro \& Shields \cite{Shapiro61} summarized in the following theorem.
\begin{thm}[Carleson \cite{Carleson1958}, Shapiro \& Shields \cite{Shapiro61}]
Let $\cZ:=\{z_i\}$ a sequence in the unit disc. Then, the following are equivalent. 
\begin{enumerate}
    \item $\cZ$ is interpolating for $H^\infty$,
    \item $\cZ$ is universally interpolating for $H^2$,
    \item $\cZ$ is onto interpolating for $H^2$,
    \item $\cZ$ is separated in the hyperbolic metric and 
    \[ \sum_{i=1}^\infty |f(z_i)|^2(1-|z_i|^2) \leq C_\cZ \norm f \norm_{H^2}^2, \,\,\, \forall f \in H^2, \tag{$CM_{H^2}$} \label{CMH2} \] for some positive $C_\cZ$ depending only on $\cZ.$
\end{enumerate}
\end{thm}
From the above conditions, probably the most enigmatic one is \eqref{CMH2}. Another way to formulate it is by saying that the atomic measure 
\[ \mu_\cZ := \sum_{z\in\cZ} \delta_z (1-|z_i|^2) \]
is a {\it Carleson measure} for $H^2$, i.e.
\[  H^2 \subseteq L^2(\mu_\cZ, \DD).\]
Such measures have a neat characterization due to Carleson. Let $z\in \DD$ we denote by $S(z)$ the region enclosed by the unit circle $\TT$ and the hyperbolic geodesic which passes from $z$ and is perpendicular to the radius which is defined by $z$. Then a measure $\mu$ is Carleson if and only if there exists $C(\mu)>0$ such that 
\[ \mu(S(z)) \leq C(\mu) (1-|z|), \,\,\, \forall z\in \DD. \]
This completes the picture of interpolation for  $H^2$ and $H^\infty.$
At this point it is important to mention that Shapiro and Shields consider similar weighted interpolation problems for the whole range of Hardy spaces $H^p, 1<p<\infty$, but for the purposes of this paper we are interested only in the Hilbert space case.

\subsection{Reproducing kernel Hilbert spaces }

As noted already by Shapiro and Shields \cite{Shapiro61} similar interpolation problems make sense in a much more general setting. Most of the material in this section can be found in \cite{AglerMcCarthy00}.

Suppose we are given a Hilbert space $\cH$ of analytic functions on the unit disc. We say that it has a reproducing kernel if there exists a $k:\DD \times \DD \mapsto \CC$ such that \[k_z:=k(\cdot,z) \in \cH, \,\,\, \forall z\in \DD \] and 
\[ f(z)=\inner{f}{k_z}_\cH, \,\,\, \forall f\in \cH. \]
To such a space we can associate the {\it multiplier algebra} \[ \cM(\cH):= \{ g:\DD\mapsto \CC : gf\in \cH, \,\,\, \forall f\in\cH \}. \]
It can be proven that equipped with the natural norm, $\cM(\cH)$ is a commutative Banach subalgebra of $H^\infty.$ A sequence $\{z_i\}$ is usually called multiplier interpolating, or interpolating for $\cM(\cH)$  if the restriction operator 
\begin{align*}  T_\cM : & \cM(\cH) \mapsto \ell^\infty \\
& g \mapsto \{g(z_i)\} \tag{MI}
 \end{align*}
is surjective. 

In analogy to the Hardy space  we can define a weighted restriction operator on $\cH$ as follows 
\begin{align*}
    T_\cH : & \cH \dashrightarrow \ell^2 \\
   & f \mapsto \Big\{ \frac{f(z_i)}{\norm k_{z_i}\norm_\cH} \Big\}.
\end{align*}
The dashed arrow indicates the fact that without further assumptions on $\{z_i\}, T_\cH $ might not be defined everywhere.

\begin{defin}Let $\cZ = \{ z_i \}$ a sequence. We say that $\cZ$ is 
\begin{itemize}
    \item {\it Onto interpolating} (OI) if $T_\cH$ is surjective,  \label{OIH} 
    \item {\it Universally interpolating } (UI) if it is onto interpolating and $T_\cH$ is bounded as linear operator. \label{UIH}
\end{itemize}
\end{defin}
The boundedness condition can be tautologically translated to a Carleson measure condition, specifically that the measure 
\[  \mu_\cZ : = \sum_{z\in \cZ} \frac{\delta_z}{\norm k_z \norm_\cH^2} \tag{CM}\]
is a Carleson measure for $\cH$, i.e. $\cH \subseteq L^2(\DD,\mu_\cZ).$

The separation condition appearing in the theorem of Carleson is a more delicate matter. We shall say that the sequence $\cZ$ is {\it weakly separated} if there exists $\varepsilon > 0 $ such that 
\[ \frac{|\inner{k_{z_i}}{k_{z_j}}|^2}{ \norm k_{z_i} \norm^2 \norm k_{z_j}\norm^2}  < 1-\varepsilon, \,\,\, \forall i \neq j \tag{WS}.\]

In the literature appears also a stronger notion of separation. We say that $\cZ$ is {\it strongly separated} if there exists $C_\cZ$ and $\{f_i \} \subseteq \cH$ such that 
\begin{align*}
    & \norm f_i \norm \norm k_{z_i} \norm \leq C_\cZ \\
    & f_i(z_j)=\delta_{ij}, \,\,\, \forall i,j. \tag{SS} \label{StrongSep}
\end{align*}

It is by now known that in a large class of spaces, including the Hardy space, the Dirichlet space (see next section) and certain weighted versions of them 
that 
\[ (MI) \iff (UI) \iff (CM) + (WS).  \]

In fact in a recent breakthrough paper \cite{Aleman17} Aleman, Hartz, McCarthy and Richter prove this result for all spaces satisfying the so called {\it complete Nevanlinna Pick property} \footnote{Here we will not enter into the details about spaces with the complete Nevanlinna Pick property but the interested reader can find more information on such spaces in the monograph \cite{AglerMcCarthy00}.}. The Dirichlet space is one of the most prominent examples of spaces satisfying this property \cite{Shapiro1962}. We can therefore say that we understand universally interpolating sequences very well. The problem reduces in each concrete case to the problem of describing the Carleson measures of the space. See also \cite{Boe2005} for developments regarding this problem prior to \cite{Aleman17}.

\subsection{Onto interpolation and the Dirichlet space}

The situation regarding onto interpolating sequences is more intricate. For one thing, in the Hardy space,  such sequences are automatically universally interpolating. On the other hand, one would be wrong to believe that this happens in every complete Nevanlinna Pick space. It is time to  introduce the main actor of this paper.

The Dirichlet space $ \mathcal{D} $ is defined as the space of analytic functions $ f $ in the unit disc $ \DD $ such that 
\begin{equation*}
\normD{f}^2 := \normH{f}^2+\int_{\DD}|f'|^2dA < + \infty,
\end{equation*} where $ \normH{\cdot} $ is the Hardy norm and $ dA=dxdy/\pi $.  It can be veryfied that with this norm $ \mathcal{D} $ is a reproducing kernel Hilbert space, with reproducing kernel 
\begin{equation*}
k_z(w)=k(w,z)=\frac{1}{w\overline{z}}\log\frac{1}{1-w\overline{z}}, \,\, z,w\in\DD
\end{equation*} which has the complete Nevanlinna Pick property \cite[p. 58]{AglerMcCarthy00}. The norm of the kernel vectors is 
\[  \norm k_z \norm_\cD^2 = \frac{1}{|z|^2}\log\frac{1}{1-|z|^2}=:d(z).\]
If we write $d_h(\cdot,\cdot)$ for the hyperbolic distance in the unit disc
\[ d_h(r,0) = \log_2\frac{1+r}{1-r}, \,\,\, 0<r<1, \]
one can see that 
\[ d(z) \approx d_h(0,z)+1. \]
 The first to study interpolation problems in the Dirichlet space have been  Bishop \cite{Bishop94} and Marshall \& Sundberg \cite{Marshall94}. Their work, unfortunately, remains unpublished but most of their results can be found also in other sources. As it is to be expected universally interpolating sequences are characterized by the weak separation condition and the Carleson measure condition. In this concrete situation  weak separation is equivalent to say that there exists $ \varepsilon>0$ such that 
 \[ d_h(z_i,z_j) \geq\varepsilon (d_h(z_i,0)+1) ,\,\,\, \forall i\neq j. \tag{WS$_\mathcal{D}$} \label{WeakSep} \] 
 About strong separation what can be said is that it is equivalent to the fact that there exists a sequence $\{m_i\} \in \cM(\cD)$, uniformly bounded in the multiplier norm, such that 
 \[ m_i(z_j) = \delta_{ij}, \,\,\, \forall i,j. \]
 
 The classical characterization of Carleson measures for the Dirichlet space in terms of logarithmic capacity given by Stegenga \cite{Stegenga80} is the following. We denote by $c$ the logarithmic capacity of compact subsets of $\TT$ and we also adopt the notation $I_z:= \overline{S(z)}\cap \TT$. Then a measure is Carleson  for the Dirichlet space if and only if it satisfies the following {\it sub-capacitary condition}, i.e.  there exists $C_\mu > 0 $ such that  for any  $z_1, \dots z_k \in \DD$
 \[ \mu\Big( \bigcup_{i=1}^k S(z_i) \Big)\leq C_\mu c\Big( \bigcup_{i=1}^kI_{z_i} \Big). \tag{SC} \label{SubCap}\]

Already Bishop notes that if the measure associated to the sequence satisfies the one box sub-capacitary condition ($k=1$ in \eqref{SubCap}) the sequence is onto interpolating, and he constructs a sequence which is onto but not universally interpolating (looking back it is clear that the same result is implicit in the work of Marshall and Sundberg). The one box sub-capacitary condition was also noted by B\o e \cite[Corollary 5.1]{Boe}. Bishop also proves that strong separation is equivalent to onto interpolation.

Another contribution comes from the work of Arcozzi Rochberg and Sawyer \cite{Arcozzi16}, which can be found in a published form in \cite{Arcozzi2020}. They prove that not only the one box sub-capacitary condition together with weak separation is not necessary for onto interpolation but they even construct sequences $\cZ$ such that the associated measure is {\it infinite}. 

To see why this result is somewhat surprising, it helps consider the connection with zero sets in the Dirichlet space. A sequence $\cZ$ is called a {\it zero set} for the Dirichlet space if there exists $f\in\cD$ not identically zero, such that $f(z)=0, \forall z\in\cZ.$ A characterization of zero sets is a notoriously difficult problem (see \cite{Carleson1952} \cite{Shapiro1962} \cite{Richter2004} \cite{Kellay2011} and \cite{Mashreghi2010}). Nonetheless,  every onto interpolating sequence $\{z_i\}$ is automatically a zero set. That is because we can find an $f\in\cD$ such that $f(z_0)=1, f(z_i)=0, \forall i\geq 1 $, then the Dirichlet function $(z-z_0)f(z)$ is not identically zero and it vanishes on $\cZ$.

One of the most general sufficient criteria is the following due to Shapiro \& Shields \cite{Shapiro1962}. A sequence $\{z_i\}$ is a zero set for $\cD$ if 
\[ \sum_{i=1}^\infty \frac{1}{d(z_i)} < +\infty, \]
Or in our language, has a finite associated measure \footnote{In fact the analogous condition in the Hardy space $H^2$  is exactly the Blaschke condition which characterized completely the zero sets in the Hardy space}. Furthermore this result is sharp in the sense that any other sufficient criterion for a zero sequence must  depend not only on $\{ |z_i| \}$ but also on their argument \cite{Nagel1982}.  So, if a sequence has infinite associated measure is not always clear if it is a zero set, let alone an onto interpolating sequence. 

\subsection{Main results}
In this direction we prove a capacitary characterization of strongly separated sequences and therefore onto interpolating sequences, in the Dirichlet space, under the additional assumption that the associated measure is finite. Due to the aforementioned results of Arcozzi Rochberg \& Sawyer \cite{Arcozzi16} this does not constitute a full characterization of onto interpolating sequences.

Our result involves an interesting condenser capacitary condition which we will now discuss. 

Suppose $z\in \DD\setminus \{0\}$, as we saw earlier the Carleson boxes $S(z)$ fit well with the geometry of the Hardy space, but often for the geometry of the Dirichlet space space one needs to modify them. Let $0<\eta\leq1$ and $z^*:=z/|z|$ we define the {\it  blow up}  of the Carleson box  
\[ S^\eta(z):=\{ w\in \DD : w\in S(z^*(1-(1-|z|)^\eta), |w| \geq |z| \}. \]
Consistently with our previous notation we write $I^\eta_z := \overline{S^\eta(z) \cap \TT}$. Note that everything reduces to the standard situation when $\eta =1 $. We denote by $\Delta_r(z)$ the hyperbolic disc or hyperbolic radius $r$ centered at $z$.  See also Figure \ref{Configuration}.

As mentioned already our condition involves the capacity of a condenser $\capacity_\DD(E,F)$ for two disjoined $F_\sigma$ subsets of $\overline{\DD}$. Its definition is classical but we recall it in Section \ref{CapacityCondensers}.

We can now formulate our condition.  We say that a weakly separated sequence $\cZ := \{z_i\} $ satisfies the \textit{capacitary condition} if there exist constants $ K>0, \gamma<1 $, depending only on $\cZ$ such that,
 
 \begin{equation*}\tag{CC}\label{CAPACITARY}
 \capacity_\DD \Big( \Delta_1(z_i) , \bigcup_{z \in S^\gamma(z_i)\cap \cZ}I_{z}  \Big) \leq  \frac{K}{d(z_i)}, \,\,\, \forall z_i\in \cZ.
 \end{equation*}
 The meaning of this otherwise obscure inequality is physically quite simple. If one considers  a condenser with one plate a hyperbolic disc of constant radius around a point of the sequence, and  as second plate the union of the intervals $I_z$ for all other points in the sequence in the ``vicinity" of $z_i$, then an electric charge of one unit in one plate, creates an electric field of total energy which bounds asymptotically  the hyperbolic distance of $z_i$ to the origin. 
 
 The plates of the condenser in Figure \ref{Configuration} are marked with grey and bold intervals respectively.
 
 \begin{figure}[t] 
\includegraphics[scale = 3.5]{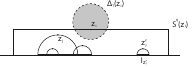}
\centering
\caption{A possible configuration of points.}
\label{Configuration}
\end{figure}

\begin{theorem}\label{Theorem A}Let $ \{z_i \} $ be a sequence in the unit disc which has finite associated measure, i.e. $ \sum_{i=1}^{\infty}1/d(z_i)<+\infty $. Then, $ \{z_i\} $ is onto interpolating for the Dirichlet space iff it is weakly separated and satisfies the capacitary condition.

\end{theorem}

The proof of Theorem \ref{Theorem A} is constructive in the sense that for given data we construct a function which solves the interpolation problem. In the literature there are two main ways to construct Dirichlet functions which solve interpolation problems, either of universal or onto type. They are both based on some kind of building blocks but the constructions are quite different. The first one, initiated by B\o e in \cite{Boe} was used to solve the universal interpolation problem in Besov spaces, and later exploited further by Arcozzi Rochberg and Sawyer in \cite{Arcozzi16} to give sufficient conditions for onto interpolation in the Dirichlet space. It should be mentioned that traces of B\o e's construction can be found in the work of Marshall and Sundberg  \cite{Marshall94}. The second construction is due to Bishop \cite{Bishop94}, and makes use of conformal mappings. In this work we combine both approaches and we construct building blocks that have the best of both worlds, in the sense that the relevant error terms arising are easier to control. We should also mention that the abstract approach of Aleman Hartz McCarthy and Richter \cite{Aleman17}, does not seem to be able to give any advantage in this concrete situation.

Another feature of our construction is that we use an iterative scheme of interpolation which is based on a quantitative version of the Theorem 
\[ (SS) \iff (OI) \]
of Bishop \cite{Bishop94}. 
To be more precise we should first quantify the conditions $(SS)$ and $(OI)$. Let $\cZ$, as always, be a sequence in the unit disc, we define its {\it strong separation constant}, denoted by $StrongSep(\cZ)$ as the infimum of all $C_\cZ>0$ such that $(SS)$ holds. Similarly for weak separation, we call weak separation constant the supremum over all $\varepsilon$ such that \ref{WeakSep} holds. If $\cZ$ is onto interpolating an application of the closed graph theorem provides a constant $C_\cZ>0$ such that for an $\alpha \in \ell^2$ there exists $f\in \cD$ such that 
\begin{align*}
    T_\mathcal{D} f & =\alpha \\
 \normD{ f } & \leq C_\cZ \norm\alpha \norm_{\ell^2}.
\end{align*}Again the infimum over all such constants we call it the {\it onto interpolation constant} of $\cZ$ and we write $OntoInt(\cZ).$
We are justified therefore to call the next theorem a quantitative version of Bishop's Theorem. 

\begin{theorem}\label{Theorem B}
    Let $K_0>0$. If a sequence $\cZ$ satisfies $StrongSep(\cZ) \leq K_0$, then 
    \[OntoInt(\cZ) \leq C_{K_0}, \]
    where $C_{K_0}$ depends only on $K_0$ and not on $\cZ$.
\end{theorem}

The proof of this theorem depends largely on the original proof of Bishop \cite{Bishop94}, but it requires a careful extraction of the relevant constants. 
\subsection{Connections with non analytic interpolation in \texorpdfstring{$H_1(\DD)$}{H1(D)}}
Next we consider the problem of onto interpolation in the Sobolev space $ H_1(\DD) $, the space of $ L^2(\DD) $ functions on the unit disc with weak partial derivatives of first order also in $ L^2(\DD) $. In this space pointwise evaluations are not well defined, therefore the definition of interpolation has to be somewhat different.  We shall say that a sequence $ \{ z_i \} $ is onto interpolating for $ H_1(\DD) $ if there exists $ \varepsilon>0 $ such that for any $ \alpha = \{ a_i \} \in \ell^2(\NN) $, there exists $ u\in H_1(\DD) $ such that $ u|_{\Delta_\varepsilon(z_i)} \equiv \sqrt{d(z_i)} \cdot a_i $. We choose the weights $ d(z_i) $ in the definition of (OI) sequences for $ H_1(\DD) $ in analogy with the holomorphic case.

In this case we have a complete characterization of onto interpolating sequences.

\begin{theorem} \label{Theorem C}
A sequence $ \{z_i\} \subseteq \DD $ is onto interpolating for $ H_1(\DD) $ iff it is weakly separated and satisfies the capacitary condition. 

\end{theorem}

From this theorem we can derive a useful corollary.

\begin{colo}
Suppose that $\cZ$ is a sequence of points in the unit disc with finite associated measure. Then it is onto interpolating for the Dirichlet space if and only if it is onto interpolating for $H_1(\DD)$.
\end{colo}

\subsection{Other results about onto interpolation}
In order to illustrate the power of our results we will prove a sufficient condition for onto interpolation which generalizes the so called weak simple condition\footnote{Not to be confused with the weak separation condition which in \cite{Arcozzi16} is called just ``separation".} of Arcozzi Rochberg and Sawyer \cite[Theorem A]{Arcozzi16} which in its turn generalizes the Bishop's one box subcapacitary condition.

In analogy with the weak simple condition of Arcozzi et al, given a sequence $\{z_i\}$ and $\gamma<1 $ we shall say that $z_i$ has {\it $\gamma$-uninterrupted view} of $z_j$ if there exists no other point $z_k \in S^\gamma(z_i)$ in the sequence such that $S^\gamma(z_k) \supseteq S(z_j) .$ Therefore the next theorem implies \cite[Theorem A]{Arcozzi16} for $\gamma=1.$

\begin{theorem}\label{THEOREM D}
	Let $ \{ z_i \} $ a weakly separated sequence with constant of weak separation $\varepsilon>0$. Assume also that there exists a  $\gamma\in(1-\varepsilon,1]$ such that  \begin{equation*}
\sum\frac{1}{d(z_j)} \leq \frac{C}{d(z_i)},
\end{equation*}  where the sum is taken over all $z_j$ in the sequence such that $z_i$ has $\gamma$ - uninterrupted view of $z_j$. Then it satisfies the capacitary condition.
\end{theorem}

 In analogy with the case of universal interpolation, a natural property for the capacitary condition that one could ask is to respect unions. More precisely, if $ \{z_i\}, \{w_i\} $ are universally interpolating sequences such that their union is weakly separated, then $ \{z_i\}\cup\{w_i\} $ is also universally interpolating, simply because the sum of two Carleson measures is a Carleson measure. We prove that not only this is not true in the case of onto interpolating sequences but we have the following stronger failure. 
 
 \begin{theorem}\label{Thoerem F}
 	There exist sequences $ \{z_i\}, \{w_i\} $ in $ \DD $ such that $ \{z_i\} $ is universally interpolating and $ \{w_i\} $ is onto interpolating for the Dirichlet space both with finite associated measures, their union is weakly separated but it is not an onto interpolating sequence.
 \end{theorem}
 
\subsection{Organization of the paper}
Section \ref{CapacityCondensers} is a collection of definitions and known results together with some elementary estimates on capacities of condensers that will be used throughout. In Section \ref{SobolevInterpolation}  we give the proof of Theorem \ref{Theorem C}, that introduces some of the techniques that will be used later without involving the complications of analyticity. In Section \ref{BishopTheorem} we present a proof of the quantitative version of Bishop's Theorem . In Section \ref{DirichletInterpolation}   using the quantitative version of Bishop's Theorem we provide the proof of Theorem \ref{Theorem A}. Finally in Section \ref{Unions} we give the proofs of Theorems \ref{THEOREM D} and \ref{Thoerem F}.

\subsection{Notation}
For two quantities $ M,N \geq 0 $ which depend on some parameters we write $ M \lesssim N, $  if there exists some constant $ C>0 $, not depending on the parameters, such that $ M \leq C \cdot N  $. We will also write $ M \approx N $ if $ M \lesssim N  $ and $ N \lesssim M $. In statements of lemmas, propositions or theorems the dependence of constants on the parameters is denoted by subscripts. We usually write $c_0$ for an absolute constant. When we write $C$ we mean a general positive constant which might change from appearance to appearance, but it always depends on the same parameters.


\section{Condensers and Capacity.}\label{CapacityCondensers}
In this section we take a close look to the capacitary condition that appears in Theorem \ref{Theorem A}. The basic idea is that the capacitary condition can be stated equivalently in terms of logarithmic capacity. This is  the content of Proposition \ref{LogarithmicCondensers}.

Another tool we will develop in this chapter is a number of stability results. We would like to know that under certain operations on a condenser the capacity remains essentially the same.

We start with the standard definition of a plane condenser.

\begin{defin}
Let $ B $ be a Jordan domain in $ \mathbb{C} $ and $ E,F \subseteq \overline{B} $ compact disjoint sets. We call the triplet $ (B,E,F) $ a condenser with field $ B $ and plates $ E,F $. Its capacity is defined as 
\begin{equation*}
\capacity_B(E,F)=\inf_{u}\int_{B\setminus(E\cup F)}| \nabla u|^2dA
\end{equation*} where the infimum is taken over all  functions $ u\in  C(\overline{B}) $ which are uniformly Lipschitz continuous on compact subsets of $ B $ and $ u|_E=0, u|_F=1 $, we will call such functions admissible for the condenser $ (B,E,F) $. If there exists a minimizer for the Dirchlet integral, i.e. $ \capacity_B(E,F)=\int_{B\setminus (E\cup F)}|\nabla u|^2dA $ and $ u $ is admissible, we shall call it the equilibrium potential of the condenser.
\end{defin} 
 For more details on condenser capacities and equilibrium potential the reader is referred to \cite{Dubinin14}.

In case that $ E\cap F\neq \emptyset $ we employ the convention $ \capacity_\DD(E,F)=+\infty $, although we will take care so that the plates of our condensers do not intersect. In the general case that $ E,F $ are $ F_\sigma $ sets we consider $ E_n, F_n $ increasing sequences of sets such that $ \cup_nE_n=E, \,\, \cup_n F_n = F $ and we define $ \capacity_\DD(E,F):=\lim_n\capacity_\DD(E_n,F_n) $.

The capacity of a condenser is  conformaly invariant in the sense that if $ \xi $ is holomorphic in $ B $ and continuous and injective on $ \overline{B} $, $ \capacity_B(E,F)= \capacity_{\xi(B)}(\xi(E),\xi(F)) $, because every admissible function $ u $ for the second condenser gives an admissible function $ u\circ \xi $ for the first one with the same energy and vice versa. 

For a set $ E\subseteq \TT $ we shall write $ c(E) $ for its logarithmic capacity. For our purposes logarithmic capacity is defined as follows 

\begin{equation*}
c(E):=\capacity_\DD(\Delta_1(0),E).
\end{equation*}

It can be proven \cite{Adams} that this definition gives rise to a capacity which is comparable to the standard logarithmic capacity that one can find in the literature.

 We now turn to the condensers appearing in the capacitary condition and some variants. Suppose that we have a base point $ z\in\DD $ and a finite sequence of points $ z_1,\dots z_N \in\DD $. One can associate a number of condensers to this configuration of points. We are interested in three types of condensers
\begin{align*}
& (\DD,\Delta_1(z), \bigcup_{j=1}^{N} S(z_j)),  \\
& (\DD,\Delta_1(z), \bigcup_{j=1}^{N} \Delta_1(z_j) ),    \\
& (\DD, \Delta_1(z), \bigcup_{j=1}^{N}  I_{z_j} ).
\end{align*}

Some justification is necessary. First of all let us mention that although such type of condensers do not appear explicitly in the literature, one can trace this construction back in the work of Bishop \cite[Theorem 1.2]{Bishop94}, although the condensers appearing there is of ``analytic nature" meaning that the admissible functions are required to be analytic. 
On the other hand, in the work of Arcozzi Rochberg and Sawyer\cite{Arcozzi16} the authors characterize onto interpolating sequences for a Dirichlet type space defined on a tree in terms of a discrete condenser capacity reminiscent of our definition (see the tree capacitary condition \cite[p.6]{Arcozzi16}).

Initially we will work with condensers of the third type, but as it turns out, under the separation hypothesis all condensers have comparable capacity.

\subsection{Condensers and capacity blow up}

We proceed now to the first of our stability results. Suppose we are given an arc $ I\subseteq \TT, \kappa>0, 0<\eta\leq1 $ then we define $ \kappa\cdot I^\eta $  as the arc having the same midpoint with $ I $ and length $ \kappa|I|^\eta $. Then, in general if $ G\subseteq \TT $ is an open set
we define the ``blow up" $ \kappa\cdot G^\gamma $ naturally as
\[ \kappa\cdot G^\eta := \bigcup_{I\subseteq G}\kappa\cdot I^\eta.  \]

Note that  when $ \eta<1 $  the ``exponential blow up" due to $ \eta $ has a far bigger effect that the ``scalar blow up" due to $\kappa$. The following observation is due to Bishop \cite{Bishop94}. Another proof of this fact using potential theory on trees exists implicitly in \cite[Lemma 2.7]{Arcozzi10}. 
\begin{lem}\label{BlowUp}
	Let $ G \subseteq \TT $  an open set  $ \kappa>0 $, $ 0<\eta<1 $. There exists a constant $ C_{\kappa,\eta}>0$ such that
	\begin{equation*}
	c(\kappa \cdot G ^\eta) \leq C_{\kappa,\eta} c(G).
	\end{equation*}
\end{lem}
In the next lemma $\omega$ stands for the harmonic measure.

\begin{lem}\label{HarmonicMeasureBlowup}
	Let $ I \subseteq \TT $, $ z\in \DD, |z|\geq 1/2 $ and $ 0<\eta<1 $. If $ |I|^\delta \leq 1-|z| $ for some $ 0<\delta<\eta<1 $ then, 
	
	\begin{equation*}
	\omega(z,I^\eta,\DD) \leq C_{\delta,\eta} \omega(z,I,\DD)^\alpha,
	\end{equation*} for some $ \alpha>0 $ which depends on $ \delta $ and $ \eta $ but not on $ I ,z  $. In fact the estimate is true if we choose $ \alpha=\frac{\eta-\delta}{1-\delta} $.
\end{lem}

\begin{proof}
	Without loss of generality we can assume that $ I=[0,\sigma]:=\{ \eit{2\pi\theta}: 0\leq\theta \leq \sigma \} $. Since $ I^\eta \subseteq [0,\sigma^\eta]\cup [\sigma-\sigma^\eta, \sigma] := I_+^\eta \cup I_-^\eta $ it suffices to prove the inequality only for the interval $ I_+^\eta $.
	
	Now let $ z=r \eit{\theta}  $ as in the statement. We can write $ 1-r=\sigma^\rho $ for some $ 0<\rho \leq \delta < \eta $, and $ \theta = \sigma ^ x, x \in \mathbb{R} $.
	The standard estimate for the harmonic measure of an arc gives
	\begin{equation} \label{ExpandedInterval}
	\omega(z,I^\eta, \DD) \approx \int_{-\sigma^{x-\rho}}^{\sigma^{\eta-\rho}-\sigma^{x-\rho}} (1+s^2)^{-1}ds.
	\end{equation}
	And similarly
	\begin{equation} \label{StandardInterval}
	\omega(z,I, \DD) \approx \int_{-\sigma^{x-\rho}}^{\sigma^{1-\rho}-\sigma^{x-\rho}} (1+s^2)^{-1}ds.
	\end{equation}
	We have to distinguish two cases. First consider the case $ x \leq \rho $. Since $ \sigma^{\eta-\rho}-\sigma^{x-\rho} \leq 0 $, estimate (\ref{ExpandedInterval}) becomes 
	\begin{equation*}
	\omega(z,I^\eta_+,\DD) \lesssim \frac{\sigma^{\eta-\rho}}{1+(\sigma^{\eta-\rho}-\sigma^{x-\rho})^2} \lesssim \sigma^{\eta+\rho-2x}.
	\end{equation*} In a similar fashion 
	\begin{equation*}
	\omega(z,I,\DD)^\alpha \gtrsim \frac{\sigma^{\alpha(1-\rho)}}{(1+\sigma^{2(x-\rho)})^\alpha} \gtrsim \sigma^{\alpha(1+\rho-2x)}.
	\end{equation*}The last quantity is always bigger than $ \sigma^{\eta+\rho-2x} $ if $ \alpha = \frac{\eta-\delta}{1-\delta}>0 $.
	
	For the remaining case $ x>\rho $, first we estimate $ (1+s^2)^{-1} $ by $ 1 $
	and we get $\omega(z,I^\eta_+,\DD) \lesssim \sigma^{\eta-1}.$ For the reverse estimate for $ \omega(z,I,\DD) $ we estimate again in the simplest way, because in that case $ [-\sigma^{x-\rho},\sigma^{1-\rho}-\sigma^{x-\rho}] \subseteq [-1,1] $,  and since $ (1+s^2)^{-1} \geq \frac{1}{2} $ on this interval
	\begin{equation*}
	\omega(z,I,\DD)^\alpha \gtrsim  \sigma^{\alpha(1-\eta)} \geq \sigma^{\eta-\rho}.
	\end{equation*} The last inequality is true for all $ 0\leq \rho<\delta $ if $ \alpha = \frac{\eta-\delta}{1-\delta} $.
	
\end{proof}

\begin{prop} \label{CondensersBlowUp}
	Let $ z,z_i\in\DD, \,\,\, i\in\NN, |z|\geq 1/2 $ and suppose that there exist $ \alpha>1,0<\beta<1 $ such that $ (1-|z_i|)^\beta \leq (1-|z|)^\alpha $. Then, 
	\begin{equation*}
	\capacity_\DD\Big( \Delta_1(z), \bigcup_{i=1}^{\infty}I_{z_i}^\beta \Big) \leq C_{\alpha,\beta}\cdot  \capacity_\DD\Big(\Delta_1(z), \bigcup_{i=1}^\infty I_{z_i}    \Big).
	\end{equation*}
\end{prop}

\begin{proof}
	Let us denote by $\phi_z$ the disc automorphism which interchanges $0$ and $z$. By Lemma \ref{HarmonicMeasureBlowup}, there exist constants $ C,\eta>0 $, depending only on $ \alpha, \beta $, such that $ |\phi_z(I^\beta_{z_i})| \leq C \cdot |\phi_z(I_{z_i})|^\eta $. Since $ \phi_z(I_{z_i}) \subseteq \phi_z(I^\beta_{z_i}) $, we get that $\phi_z(I^\beta_{z_i}) \subseteq C \cdot\phi_z(I_{z_i})^\eta $.
	
	In this case we can estimate as follows, for $ N\in \NN $ fixed.
	
	\begin{align*}
	\capacity_\DD\Big( \Delta_1(z), \bigcup_{i=1}^N I^\beta_{z_i} \Big) & = \capacity_\DD \Big( \Delta_1(0), \bigcup_{i=1}^N\phi_z(I^\beta_{z_i})  \Big)  \\
	& \leq \capacity_\DD \Big( \Delta_1(0) , \bigcup_{i=1}^N C\cdot \phi_z(I_{z_i})^\eta \Big) \\
	& \leq C \capacity_\DD\Big( \Delta_1(0), \bigcup_{i=1}^N \phi_z(I_{z_i})  \Big) \\
	& = C \capacity_\DD \Big( \Delta_1(z), \bigcup_{i=1}^N I_{z_i} \Big).
	\end{align*} 
	
	The result follows by letting $ N $ go to infinity.
\end{proof}


\subsection{Hyperbolic geometry in the disc and stability of Carleson boxes under automorphisms of the unit disc}

One obstacle we will have to overcome when dealing with the capacitary condition is the fact that it involves an intrinsically conformally invariant quantity (the condenser capacity) and a geometric object (Carleson box) which is not  defined in terms of hyperbolic geometry of the disc. 

A manifestation of this phenomenon is in the following observation. Suppose we consider  an arc $ I_w \subseteq \TT $ corresponding to a point $ w $ in the unit disc. Then  in general, under a disc automorphism $\phi_w $ 
\[ \phi_z(I_w) \neq I_{\phi_z(w)}. \]

One way to get around this problem is to ask for the point $ z $ to be closer to the origin than $ w $ is. In such a case we can expect a stability of the geometry of Carleson boxes under a disc automorphism.

\begin{lem} \label{CarlesonBoxStability}
	Let $ z\in \DD $ and 
	\[ G = \bigcup_{i} I_{w_i} \]
	an open set. Suppose also that 
	\[  |z| \leq |w_i|, \,\,\, \forall i. \]
	Then, there exists an absolute constant $ \kappa>0, $ such that 
	\[ \frac{1}{\kappa}\cdot \phi_z(G) \subseteq \bigcup_{i}I_{\phi_z(w_i)} \subseteq \kappa \cdot \phi_z(G).\]
\end{lem}
\begin{proof}
	Notice that since an automorphism of the unit disc extends to a homeomorphism on the boundary it suffices to prove the claim when $ G $ is a single interval $ I_w $. 
	
	It suffices to prove the claim when $ |z|\geq 1/2 $. Even more, it is always true that  $ \phi_z(I_w) \cap I_{\phi_z(w)} \neq \emptyset $, therefore our claim will follow if we prove that their lengths are comparable. To this end, consider a $ \zeta \in I_w $ and suppose first that $ |z^*-\zeta| \geq 2\pi (1-|z|) $. Consequently,
	\begin{align*}
	|z^*-w^*| & \geq |z^*-\zeta|-|\zeta-w^*| \\
	& \geq |z^*-\zeta|-\pi(1-|w|) \\
	& \geq |z^*-\zeta|-\pi(1-|z|) \\
	& \geq \pi|z^*-\zeta|.
	\end{align*}
	
	Hence, for $ z,w \in\DD $ and $\zeta\in I_w $ as before, we have that \begin{equation*}
	|1-z\overline{w}| \approx \max \{1-|z|,1-|w|,|z^*-w^*| \} \gtrsim \max \{1-|z|,|z^*-\zeta| \} \approx |1-z\overline{\zeta}| . 
	\end{equation*}
	But this last estimate remains true also in the case $ |z^*-\zeta| \leq 2\pi(1-|z|) $.
	
	As a matter of fact, the converse inequalities follows by a similar consideration, examining the cases $ |z^*-w^*| \geq 2\pi(1-|w|)$ and $ |z^*-w^*| \leq 2\pi(1-|w|) $.
	
	Hence,
	\begin{equation*}
	|\phi_z(I_w)| = \int_{I_w}\frac{1-|z|^2}{|1-\overline{\zeta}z|^2}|d\zeta| \approx \frac{(1-|z|)(1-|w|)}{|1-\overline{w}z|^2} \approx |I_{\phi_z(w)}|.
	\end{equation*}

\end{proof}

\subsection{Stability of condenser capacity under perturbation of plates}

In this point we introduce a tool which proves to be critical for our constructions. We introduce it here because it will come handy in the proof of the next lemmas, but we will use it again in Section \ref{BishopTheorem} as a building block for our interpolating functions. 

First a bit of notation. For an interval $I\subseteq \TT$ we write $S(I)$ for the Carleson box $S(w)$ such that $I_w=I$ and also if $G\subseteq \TT$ is an open set on the circle we use the notation
\[ S(G):= \bigcup_{I\subseteq G} S(I). \]

Let now $G\subseteq \TT$ then there exists an equilibrium measure $\mu_G$ for $G$ (as defined for example in \cite[p.19]{Mashreghi14}) and an associated holomorphic potential defined as 
\[ \varphi_G(z):=\int_{\TT}\log\frac{e}{1-z\overline{\zeta}}d\mu_G(\zeta). \]
This function has some useful properties.
\begin{prop} \cite[Lemma 2.3]{Cascante2012}\label{LogPotential}
 Let $G$ and $\varphi_G$ as before, then the following is true.
 \begin{enumerate}
     \item $|\Im \varphi_G(z)| \leq \frac{\pi c(G)}{2}, \,\,\, \forall z\in \DD $,
     \item $ 0 \leq \Re (\varphi_G) (z) \leq 1 , \,\,\, \forall z \in \DD,$
     \item $|\varphi_G| \leq \frac{\pi}{2} \Re( \varphi_G)$,
     \item $|\varphi_G(z)| \geq \varepsilon, \,\,\, \forall z\in S(G)$,
     \item $\normD{\varphi_G}^2 \leq c_0 c(G)$,
     \item $\varphi_G$ is univalent.
 \end{enumerate}
\end{prop}
The fact that it is univalent comes from the observation that it has a derivative with positive real part.

\begin{lem}\label{StrechingCapacity}
	Suppose that $ (\DD,E,F) $ is a condenser and $ 0<a<b $. Also $ u \in H_1(\DD) \cap C(\overline{\DD}) $ such that $ u\leq a $ in $ E $, $ u \geq b $ in $ F $. Then \begin{equation*}
	\capacity_{\DD}(E,F) \leq \frac{1}{(b-a)^2}\int_{\DD}|\nabla u|^2dA.
	\end{equation*}
	
\end{lem}
\begin{proof}
	Define the function $$ g:=\min\{ \max \{\frac{u-a}{b-a},0\},1 \}. $$
	Then $ g $ is an admissible function for the condenser $ (\DD,E,F) $, hence, 
	\begin{equation*}
	\capacity_{\DD}(E,F) \leq \int_{\DD}|\nabla g|^2 dA \leq \frac{1}{(b-a)^2}\int_{\DD}|\nabla u|^2 dA.
	\end{equation*}
\end{proof}

\begin{lem} \label{Betsakos}
	Suppose that $ 	w\in \DD, d_h(w,0)>2 $ and $ u $ is the equilibrium potential for the condenser $ (\DD, \Delta_1(0), \Delta_1(w)) $. Then $ u(z)\geq 1/2 $ for every $ z $ such that $ d_h(z,w)\leq d_h(z,0) $, where $ d $ is the hyperbolic distance in $ \DD $.
\end{lem}
\begin{proof}
	First consider $ \phi $ an automorphism of the unit disc such that $ \phi(0)=-r, \phi(w)=r, r>0 $. By conformal invariance, $ v:=u\circ\phi $ is the equilibrium potential for the condenser $ (\DD, \Delta_1(-r), \Delta_1(r)) $. Also by symmetry, $ v(-x+iy)=1-v(x+iy) $, therefore $ v(iy)=1/2 $. Suppose now that at some point $ z_0\in\DD, \Re(z_0)>0 $, $ v(z_0)<1/2 $. In that case the function $ h $ defined by
	\begin{equation*}
	h(z):=
	\begin{cases}
	v(z), \,\,\, \text{if} \,\,\, \Re(z)\leq 0, \\
	\max\{\frac{1/2+v(z_0)}{2},v(z)\}, \,\,\, \text{if} \,\,\, \Re(z) \geq 0,
	\end{cases}
	\end{equation*} is admissible for the condenser and has strictly smaller Dirichlet integral, which contradicts the fact that $ v $ is the minimizer.
\end{proof}

We can now prove the following.

\begin{prop}\label{ComparisonCapacities} Suppose that $ z\in \DD $ and $ z_1,\dots,z_N\in\DD $, such that $ |z_i|\geq |z| $ and $d_h(z,z_i) \geq 3 $. Then, 
	\begin{equation*}
	\capacity_{\DD} \Big( \Delta_1(z), \bigcup_{i=1}^N S(z_i) \Big) \approx \capacity_{\DD} \Big( \Delta_1(z), \bigcup_{i=1}^N \Delta_1(z_i) \Big) \approx \capacity_{\DD} \Big( \Delta_1(z), \bigcup_{i=1}^N I_{z_i} \Big).
	\end{equation*} Where the implied constants are absolute.
\end{prop}

Before going into the proof, let us remark that the assumption $d_h(z,z_i)>3$ is indeed necessary (although the particular constant is not essential) because otherwise even for $N=1$ it might happen that the capacity of the first and second condenser goes to infinity while the capacity of the third one remains bounded. Since we are interested in comparability of the capacities only when they tend to zero, this is not a major issue for us.

\begin{proof} 
	Trivially the first capacity is bigger than the third one. 
	
	To prove the other estimates we first examine the case $ z = 0 $.
	Consider the set $ E=\bigcup_{j=1}^N I_{z_j} $. Consider the equilibrium potential $\varphi_E$ and apply to it Lemma \ref{StrechingCapacity} in order to get
	\begin{align*}
	\capacity_{\DD} \Big( \Delta_1(0), \bigcup_{i=1}^N S(z_i) \Big) & \leq \frac{1}{(c-C_0 c(E))^2} \int_{\DD}|\nabla \varphi_E(z)|^2 dA(z) \\ & \lesssim c(E) \\ & = \capacity_{\DD} \Big( \Delta_1(0), \bigcup_{i=1}^N I_{z_i} \Big).
	\end{align*}
	
	Without loss of generality in the above estimate we assumed that $ c(E) $ is sufficiently small, otherwise the estimate is trivial.  
	
	Before we proceed let us note that by Dirichlet's principle \cite{Marshall94} the equilibrium potential $ u $ for the condenser $ \Big( \Delta_1(z), \bigcup_{i=1}^N \Delta_1(z_i) \Big) $ is harmonic in the domain $ \Omega :=  \DD\setminus \{\overline{\Delta_1}(z_1),\overline{\Delta_1}(z_2) \dots, \overline{\Delta_1}(z_N) \}  $.
	
	For a fixed $ i\in \{1,\dots,N \} $ let $  u_i $ be the equilibrium potential for the condenser $\capacity_{\DD} \Big( \Delta_1(0), \Delta_1(z_i) \Big) $. Then by the maximum principle 
	\[ u\geq u_i \,\,\, \text{on} \,\,\, \Omega. \]

	By appealing to Lemma \ref{Betsakos}, we get $u \geq u_i\geq 1/2 $ on $ I_{z_i} $. Hence, \begin{equation*}  \capacity_{\DD} \Big( \Delta_1(0), \bigcup_{i=1}^N I_{z_i} \Big) \leq 4 \capacity_{\DD} \Big( \Delta_1(0), \bigcup_{i=1}^N \Delta_1(z_i) \Big).
	\end{equation*}
	
	Suppose now that $ z $ is not necessarily zero. By the stability Lemma \ref{CarlesonBoxStability} we can find $\kappa > 0 $ such that 
	\[ \phi_z(\bigcup_{i=1}^N S(z_i)) \subseteq  \bigcup_{i=1}^N \kappa\cdot S(\phi_z(z_i)) \]
	Applying this we get
	\begin{align*}
	\capacity_\DD(\Delta_1(z), \bigcup_{i=1}^N S(z_i)) 
	& \leq \capacity_\DD(\Delta_1(0),  \bigcup_{i=1}^N \kappa\cdot S(\phi_{z}(z_i))) \\
	& \lesssim \capacity_\DD(\Delta_1(0), \bigcup_{i=1}^N \kappa\cdot I_{\phi_{z}(z_i)}) \\
	& \lesssim \capacity_\DD(\Delta_1(0), \bigcup_{i=1}^N \phi_{z}(I_{z_i}) ) \\ 
	& = \capacity_{\DD} \Big( \Delta_1(z), \bigcup_{i=1}^N I_{z_i} \Big).
	\end{align*}
	The remaining estimates
	\begin{equation*}
	\capacity_{\DD} \Big( \Delta_1(z), \bigcup_{i=1}^N I_{z_i} \Big) \lesssim
	\capacity_{\DD} \Big( \Delta_1(z), \bigcup_{i=1}^N \Delta_1(z_i) \Big) \lesssim 
	\capacity_\DD(\Delta_1(z), \bigcup_{i=1}^N S(z_i)).
	\end{equation*} 
	
	can be handled in much the same way, therefore the proof will be omitted.
\end{proof}

The following proposition allows us to express the condenser capacity appearing in Theorem \ref{Theorem A} in terms of logarithmic capacity.

\begin{prop}\label{LogarithmicCondensers}
	Suppose that $ z, z_1, \dots z_N \in \mathbb{D} $ such that $1-|z_i| \leq (1-|z|)/2$. Then 
	\begin{equation*}
	c \Big( \bigcup_{i=1}^N I_{\phi_{z}(z_i)} \Big)  \approx \capacity_{\DD} \Big(\Delta_1(z), \bigcup_{i=1}^N I_{z_i} \Big),
	\end{equation*}where the implied constants are absolute.
\end{prop}

\begin{proof}
	
	We have that
	\begin{align*} 
	\capacity_\DD \Big( \Delta_1(z), \bigcup_{i=1}^N I_{z_i} \Big)  & = 	\capacity_\DD \Big( \Delta_1(0), \bigcup_{i=1}^N \phi_{z}(I_{z_i}) \Big) \\ 
	&\approx 	\capacity_\DD \Big( \Delta_1(0), \bigcup_{i=1}^N I_{\phi_{z}(z_i)} \Big) \\
	&	= c\Big(  \bigcup_{i=1}^N I_{\phi_{z}(z_i)} \Big).
	\end{align*} The second estimate is justified with an argument identical to the one in the proof of Proposition \ref{ComparisonCapacities}. 
\end{proof}


\subsection{Strong separation and condensers capacity}
At this point let us introduce the following notation. If $\cZ$ is a sequence and $\gamma<1$ we write 
\[ \mathcal{V}_\gamma(z):= \{ z_j \in \cZ : S^\gamma(z)\cap S^\gamma(z_j) \neq \emptyset, |z_j| \geq |z| \}. \] This is a slightly bigger neighborhood of points than $S^\gamma(z)\cap \cZ$ but nonetheless triangle inequality gives 
\[ \mathcal{V}_\gamma(z) \subseteq S^\gamma(2 I_z) \]
As a consequence we get that the capacitary condition does not change substantially if we ask a bit more. 
\begin{lem} \label{VorSdoesnotmatter}
 If a weakly separated sequence satisfies the capacitary condition with constants $K, 1>\gamma>0$ then for $1>\eta>\gamma $ it satisfies the condition
\begin{equation*}
\capacity_\DD\Big( \Delta_1(z_i), \bigcup_{z_j\in\mathcal{V}_{\eta}(z_i) } I_{z_j} \Big) \leq \frac{K}{d(z_i)},
\end{equation*} with possibly a finite number of exceptions. 
 \end{lem}
\begin{proof}
As already noted, excluding possibly a finite number of points close to the origin 
\[ \mathcal{V}_\eta(z_i) \subseteq S(2I^\eta_{z_i}) \subseteq S^\gamma(z_i).\]
\end{proof}
The following is a simple lemma about the geometry of weakly separated sequences. A proof of it (actually a stronger statement) can be found \cite{Marshall94}. We provide a proof of the part that we are going to use for completeness.

\begin{lem}\label{SeparationLemma}
	Suppose that $ \{z_i\} $ is a weakly separated sequence, with separation constant $\varepsilon$. Then if $1 > \gamma > 1-\varepsilon$, with finite exceptions,  if $z_j \in \mathcal{V}_\gamma(z_i)$ 
	\begin{equation*}
	(1-|z_j|)^{\gamma}\leq \frac{1-|z_i|}{2}.
	\end{equation*}

\end{lem}
\begin{proof} The proof of this lemma is nothing more than the simple geometric fact that for $z\in \DD$ and $c_0>0$ the region 
\[ \{w \in \DD : |w| \geq |z|,  2(1-|w|)^\gamma \leq 1-|z|, \,\,\, |z^*-w^*| \leq c_0 (1-|z|)^\gamma  \} \]
is contained eventually (for $z$ close to the boundary) in any hyperbolic disc with center $z$ and radius  $(d_h(z,0)+1)/\varepsilon$, since $1/\varepsilon > 1/(1-\gamma).$
\end{proof}

\begin{prop} \label{NecessaryCondition} Suppose that $ \{ z_i \} $ is a $ K-$ strongly separated sequence in the unit disc. If $ \gamma<1 $ as in Lemma \ref{SeparationLemma}, there exists $ C>0 $ depending only on $K$ such that,
\begin{equation*}
\capacity_\DD\Big( \Delta_1(z_i), \bigcup_{z_j\in\mathcal{V}_{\gamma}(z_i) } I_{z_j} \Big) \leq \frac{C_{K,\gamma}}{d(z_i)},
\end{equation*}
for all but finitely many $z_i$.
\end{prop}

\begin{proof}
For a fixed point $ z_i $ then there exists a function $ f_i\in\mathcal{D} $, such that $ f_i(z_i)=1, f_i(z_j)=0$ for $ i\neq j $ and $ \normD{f_i}^2\leq K/d(z_i) $. By the standard oscillation estimate for Dirichlet functions $ |f(z)-f(w)|\leq   \normD{f} C_0 \sqrt{d_h(z,w)} $, we have that 
\begin{equation*}
|f(w)-1|\leq c_0/\sqrt{d_h(z_i,0)}, \,\,  w\in \Delta_1(z_i),
\end{equation*}and
\begin{equation*}
|f(w)|\leq c_0/\sqrt{d_h(z_i,0)}, \,\,  w\in \Delta_1(z_j), \,\, j\neq i. 
\end{equation*}
Therefore by Lemma \ref{StrechingCapacity} applied to $f_i$
\begin{equation*}
\capacity_\DD(\Delta_1(z_i),\bigcup_{j\neq i} \Delta_1(z_j)) \leq \frac{1}{(1-2C_0/\sqrt{d(z_i)})^2}\int_\DD|f_i'|^2dA .
\end{equation*}

Which gives the estimate 
\begin{equation*}
	\capacity_\DD\Big( \Delta_1(z_i), \bigcup_{j\neq i}\Delta_1(z_j) \Big) \leq 2K/d(z_i),
\end{equation*}for all but finite many $ z_i $.

Now notice that if again we exclude from the sequence a finite number of points, because of Lemma \ref{SeparationLemma}, it is true that $ z_j\in\mathcal{V}_{\gamma}(z_i)  $ implies that $ (1-|z_j|)\leq (1-|z_i|)/2  $. And therefore the result follows from Lemma \ref{ComparisonCapacities}.

\end{proof}


\section{Constructing Interpolating building blocks}
As we have already mentioned our interpolation theorems are constructive in the sense that interpolating functions are constructed  by pasting together functions that behave essentially as ``bump" functions. If one requires these bump functions to be holomorphic is clear that they cannot vanish in any non discrete set, therefore our goal is to make them very small outside a certain region, in a sense to be made precise . Here we will preset two complementary constructions. We start with the first one due to B\o e \cite[Lemma 4.1]{Boe}.

\begin{lem}[B\o e \cite{Boe}]
	Suppose that $ z\in\DD $ and $ \alpha<1 $. Then there exists a function $ f=f_{z,\alpha} $, such that
	\begin{enumerate}
	    \item $ f(z)=1 $,
	    \item $ \normD{f}^2\leq C_\alpha/d(z)$,
	    \item $\norm f\norm_\infty\leq C_\alpha $,
	    \item $ |f(w)|\leq C_\alpha e^{-6d(z_i)} $ and,
	    \item  $|f'(w)|\leq C_\alpha e^{-6d(z_i)} $ for all $ w\not\in S^\alpha(z) $.
	\end{enumerate}
\end{lem}

We will refer to the function $ f_{z,\alpha} $ as the B\o e's function associated to $ z $ and $ S^\alpha(z) $.

The second construction is essentially due to Bishop \cite[Lemma 2.13]{Bishop94}. Here we give a construction which is based on holomorphic potentials.  Bishop's original idea was to use conformal mapping and extremal distance techniques based on a result of Marshall and Garnett \cite[Theorem 4.1]{MarshallGarnett}. We shall do a very similar construction but based on logarithmic potentials instead. In many ways the two constructions are equivalent but our has some slight advantages. First, using the machinery developed in Section \ref{CapacityCondensers} we can prove a conformal invariant version of the construction and secondly we are able to control the dependence of our constants on the parameters which is crucial for proving the quantitative estimate that Theorem \ref{Theorem C} requires.

We start with a simple lemma for harmonic measure.

\begin{lem}\label{Harmonic}
Let $ I $ be a closed arc in $ \TT $. For $0<\eta<1$ we have the following elementary estimate 
\begin{equation*}
\omega(z,\TT\setminus I^\eta,\DD) = \int_{\TT\setminus I^\eta}\frac{1-|z|^2}{|\zeta-z|^2}|d\zeta| \leq C_\eta \cdot c(I),
\end{equation*} for $ z\in S(I) $.
\end{lem}

\begin{proof}
Without loss of generality $ S(I)\subseteq \{\zeta=re^{i\theta}\in\DD: 0\leq\theta \leq |I| , 1-r\leq |I| \}.$ In this case we have 

\begin{equation*}
\omega(re^{it},\TT\setminus I^\eta,\DD)  = \frac{1}{2\pi}\int_{|I|^\eta}^{2\pi}\frac{1-r^2}{|e^{is}-re^{it}|^2}ds 
\leq C \int_{|I|^\eta}^{2\pi}\frac{1-r}{|s-t|^2}ds. \end{equation*}
This last integral can be seen to be of the right magnitude 
\[\int_{|I|^\eta}^{2\pi}\frac{1-r}{|s-t|^2}ds \leq C \frac{1-r}{|I|^\eta-t}   \leq C \frac{|I|}{|I|^\eta-|I|} 
\leq C |I|^{1-\eta}   \leq C c(I). \]

\end{proof}

The following lemma is essentially the conformal invariant version of \cite[Lemma 2.13]{Bishop94}.

\begin{lem} \label{BishopFunction}
Let $ z, z_1,z_2,\dots z_N \in\DD$, such that $ 1-|z_i|\leq (1-|z|)/2 $, and 
\begin{equation*}
\capacity_\DD\Big( \Delta_1(z), \bigcup_{i=1}^NI_{z_i} \Big) \leq K/d(z).
\end{equation*}Where $ K\leq K_0 $. Set $F:=\cup_{i=1}^N I_{z_i}$. Then there exists a constant $ C $, depending only on $ K_0 $, and a function $ g_z\in\mathcal{D} $ such that 

\begin{enumerate}
\item $ g_z(z)=1 $,
\item $ \int_\DD|g_z'|^2dA \leq \frac{C}{d(z)} $,
\item $ \norm g_z\norm_\infty \leq C $
\item $ |g_z(w)| \leq Ce^{-6d(z)} $, $ w\in S(F) $,
\item $ \int_{S(F)}|g_z'|^2dA \leq C e^{-6d(z_i)} $.
\end{enumerate}
\end{lem}

\begin{proof} 
Throughout the proof $ C $ denotes a positive constant depending only on $ K_0 $. If a constant depends on some other parameter we will denote it by an index.

Fix some arbitrary $ \eta<1 $, and apply to our assumption first Proposition \ref{LogarithmicCondensers} and then Lemma \ref{BlowUp} with parameter $\eta$ to arrive at  \[ c(\bigcup_{i=1}^\infty\phi_{z}(I_{z_i})^\eta) \leq \frac{C}{d(z)} .\] 
Set $ E:=\bigcup_{i=1}^N \phi_{z}(I_{z_i})^\eta $. 
Let also $\varphi_E$ the holomorphic equilibrium potential associated to $E$ and $\psi_E:=1-(1-c(E))\varphi_E.$
Our building block will be the function 
\[ f_E:= e^{-\frac{A}{\psi_E}}, \]
where $A>0$ is a constant to be specified.  By the properties of holomorphic potentials, it is clear that 
\[ \psi_E(z) \in [c(E),1]\times  [- \frac{\pi c(E)}{2}, \frac{\pi c(E)}{2}] , \,\,\, \forall z \in \DD. \]
At this point some elementary euclidean geometry on the image of $\psi_E$ gives
\[  |\psi_E| \leq \frac{\pi}{2}\Re \psi_E.\] 
This is the crucial estimate that will allow us to derive the desired properties of $f_E.$

Clearly $ f_E $ is bounded by $1$ because the real part of the exponent is positive.

The Dirichlet integral of $f_E$ can  be estimated as follows 
\begin{align*}
\int_\DD|f_E'|^2dxdy &= A^2 \int_{\DD}\Big| \frac{\psi_E'}{\psi_E^2} \Big|^2 e^{-2A\Re\psi_E/|\psi_E|^2}dA \\ 
& \leq A^2 \int_\DD \frac { e^{-\pi A / |\psi_E|}} {|\psi_E|^4 } |\psi'_E|^2dA\\
& \leq \Big( \frac{4}{\pi e} \Big)^4 \frac{1}{A^2} \int_\DD |\psi'_E|^2 dA \\
& \leq c_0 \frac{c(E)}{A^2}.
\end{align*}
With $c_0$ absolute positive constant.

On the other hand the value of $f_E$ at the origin remains bounded below.

\begin{equation*}
|f_E(0)| = e^{-\frac{A}{1-(1-c(E))c(E)}} \geq e^{-\frac{4A}{3}}.
\end{equation*}

Suppose now that $ w\in \bigcup_{i=1}^N S(\phi_{z}(I_{z_i})) $. Let $ J_i= \phi_{z}(I_{z_i}) $. Then $ w\in S(J_i) $ for some $ i $, recalling Lemma, \ref{Harmonic}
\begin{align*}
\Re\psi_E(w) &= \int_\TT\frac{1-|w|^2}{|\zeta-w|^2}\Re\psi_E(\zeta)|d\zeta| \\  &\leq \int_{\TT\setminus J_i^\eta} \frac{1-|w|^2}{|\zeta-w|^2}|d\zeta| + c(E) |J_i^\eta|  \\
&= \omega(w,\TT\setminus J_i^\eta,\DD) + c(E) |J_i^\eta|\\ 
&\leq c_0  c(J_i) \\
&\leq c_0  c(E)  \\
&\leq  \frac{C}{d(z)}.
\end{align*}

Hence, $ |f_E(w)| \leq  e^{-\frac{A d(z)}{2 C}} $.

Finally, using the conformality of logarithmic potentials, 
\begin{align*}
\int_{\bigcup_{i=1}^NJ_i}|f_E'|^2dA &= A^2\int_{\bigcup_{i=1}^NJ_i}\frac{|\psi_E'|^2}{|\psi_E|^4}e^{-2A\Re\psi_E/|\psi_E|^2}dA(z) \\
&\leq A^2\int_{\bigcup_{i=1}^N\psi_E(J_i)}\frac{1}{|z|^4}e^{-2A\Re z/|z|^2}dA(z) \\
&\leq A^2
 \int_{0}^{c(E)}\int_{c(E)}^{Cc(E)}\frac{e^{-A/x}}{x^4}dxdy 
\\ 
& \leq C_A e^{-\frac{Ad(z)}{C}} .
\end{align*} 
 Set $ A=12C $. 
 
 We claim that  $ g_z=f_E\circ\phi_{z}/f_E(0) $ satisfies the required conditions. Properties (1)-(3) and (5) are invariant under M\"obius transformations, we get property (4) after an application of Lemma \ref{CarlesonBoxStability}.
\end{proof}

\section{Onto Interpolation in \texorpdfstring{ $ H_1(\DD) $}{H1(D)}}
\label{SobolevInterpolation}
The idea of the proof is to interpolate every value separately with  functions in $H_1(\DD)$ that have disjoint supports. The simple minded idea to take functions constant on $ \Delta_1(z_i)$ that vanish outside a bigger hyporbolic disc does not work, so we have to construct the disjoint regions in a slightly more sophisticated way.

After a series of elementary lemmas we will proceed to the proof. 

Let  $  \gamma $ as in Lemma \ref{SeparationLemma}. Then let us define the regions $ S_i $ associated to the sequence 
 
\begin{equation*}
S_i := S^{\gamma}(z_i) \setminus \bigcup_{z_j\in \mathcal{V}_{\gamma}(z_i)} S^{\gamma}(z_j).
\end{equation*}

\begin{lem}\label{Regions}
The regions $ S_i $ are pairwise disjoint and, for all but finitely many $ z_i $, we that $ \Delta_1(z_i) \subseteq S_i $ and $ \Delta_1(z_j) \subseteq \DD \setminus S_i  $ for all $ j\neq i $.
\end{lem}

\begin{proof}
Let $ i\neq j $. Without loss of generality $ |z_j|\geq|z_i| $. Hence, either $ S^\gamma(z_i)\cap S^\gamma(z_j)=\emptyset $ or $ z_j\in \mathcal{V}_\gamma(z_i) $. In both cases $ S_i \cap S_j = \emptyset $.

From Lemma \ref{SeparationLemma} it follows that $ \Delta_1(z_i) \subseteq S_i $, if $ z_i $ is sufficiently close to the boundary. Since $ S_i $ are pairwise disjoint,  $ \Delta_1(z_j)\subseteq \DD\setminus S_i $, for any $ j\neq i $.
\end{proof}

\begin{lem}\label{CofiniteW}
A sequence $ \{ z_i \} \subseteq \DD$ is onto interpolating for $ H_1(\DD) $ if a cofinite subsequence of it is.
\end{lem}

\begin{proof}
It suffice to show that if $ \{z_i \}_{i=1}^\infty $ is an onto interpolating sequence then $ \{z_i\}_{i=0}^\infty $ is.  Fix $ \varepsilon>0 $ such that $\overline{ \Delta_\varepsilon}(z_i) $ are pairwise disjoint, and such that for all $ \alpha=\{a_i\}\in \ell^2 $ there exists $ u\in H_1(\DD) $ such that $ u|_{\Delta_\varepsilon(z_i)} \equiv \sqrt{d(z_i)}a_i, i\geq 1 $. Let also $ \varepsilon ' >0 $ such that $ \overline{ \Delta_\varepsilon}(z_0) \subseteq \Delta_{\varepsilon'}(z_0)  $ and $ \overline{ \Delta_{\varepsilon'}} \cap \bigcup_{i=1}^\infty \overline{ \Delta_\varepsilon}(z_i) = \emptyset $. Then there exists $ \xi \in C^\infty(\DD) $, such that $ \xi \equiv 0 $ on $\Delta_\varepsilon(z_0) $ and $ \xi \equiv 1 $ on $ \DD \setminus \Delta_{\varepsilon'}(z_0) $. Let $ \alpha=\{a_i\}_{i=0}^\infty \in \ell^2 $ and $ u $ as before. The function 
\begin{equation*}
	v:=\xi u + a_0(1-\xi),
\end{equation*} is the interpolating function for the sequence $ \{z_i\}_{i=0}^\infty $ and the data $ \{a_i\}_{i=0}^\infty $.
 \end{proof}

\begin{thm}\label{h2Int} A sequence $ \{z_i\} \subseteq \DD $ is onto interpolating for $ H_1(\DD) $ iff it is weakly separated and 
satisfies the capacitary condition.

\end{thm}

\begin{proof}We will start with the more involved direction which is the sufficiency of the conditions in the statement. Notice that if the capacitary condition is satisfies for some $ \alpha<1 $, then it is satisfied for all $ \gamma, \alpha<\gamma<1 $. Therefore we can assume that it is satisfied for $ \gamma $ as large as the one in Lemma \ref{SeparationLemma}. Assume without loss of generality that $ \Delta_1(z_i) \cap \Delta_1 (z_j) = \emptyset, i\neq j $. The estimates that we state next might fail for a finite number of points in our sequence but in that case Lemma \ref{CofiniteW} allows us to initially disregard any finite number of points.

Then suppose that $ f_i:=f_{z_i,\gamma} $ is B\o e's function for $ z_i $ and $ S^{\gamma}(z_i) $. There exists a constant $ C_0>0 $ such that $ |f_i(z)| \leq C_0 e^{-6d(z_i)},$ for all $ z\not\in S^{\gamma}(z_i) $ and $ |1-f_i(z)| \leq C_0/d(z_i), z\in \Delta_1(z_i) $. Set 
\begin{equation*}
u_i:= \min\{\max\{\frac{|f_i|-C_0e^{-6d(z_i)}}{1-C_0/d(z_i)-C_0e^{-6d(z_i)}}  ,0\} 1 \}.
\end{equation*}

The function just constructed satisfies $ u_i|_{\Delta_1(z_i)} \equiv 1$,
$ \,\,\, u_i|_{\DD\setminus S^{\gamma}(z_i)} \equiv 0  $, $ \norm u \norm_{H_1(\DD)}^2 \leq C/d(z_i) $ and $ \norm u_i\norm _\infty \leq C $.

Next we apply Lemma \ref{ComparisonCapacities}
 to the stronger version of the capacitary condition in Lemma \ref{VorSdoesnotmatter} and we arrive at 
 \[ \capacity_\DD \Big( \Delta_1(z_i), \bigcup_{z_j\in\mathcal{V}_\gamma(z_i)} S^\gamma(z_j)  \Big)\leq \frac{K}{d(z_i)}\]
 Hence by definition of condenser capacity  there exists $ v_i\in H_1(\DD) $ such that $ v_i|_{\Delta_1(z_i)} \equiv 1 $, $ v_i|_{S^{\gamma}(z_j)} \equiv 0 $ for all $ z_j\in  \mathcal{V}_\gamma(z_i) $ and $ \norm v_i \norm_{H_1(\DD)}^2 \leq C/d(z_i) $. Without loss of generality we can also assume that $\norm v_i \norm_{\infty} \leq 1$

Our interpolation building blocks will be the functions $ w_i:=u_i\cdot v_i $. Notice that by construction $ \supp w_i \subseteq S_i $, hence, $ w_i|_{\Delta_1(z_j)} \equiv \delta_{ij} $ by Lemma \ref{Regions}. Furthermore $\norm w_i \norm^2_{H_1(\DD)} \leq C/d(z_i)$ and $\norm w_i \norm_\infty \leq C.$

This observation clearly suggests that, if $ \alpha = \{a_i\} \in \ell^2(\NN) $, the obvious candidate for interpolation is the function $ F:= \sum_{i=1}^{\infty}a_i \sqrt{d(z_i)} w_i $. It takes the right values on hyperbolic discs $ \Delta_1(z_i). $ It remains to show that is is actually in $ H_1(\DD) $.

Let $ N\in \NN $. $ F_N := \sum_{i=1}^{N}a_i \sqrt{d(z_i)} w_i $.
Then,
\begin{align*}
 \int_{\DD}&|\nabla F_N(z)|^2+ |F_N|^2(z)dA(z)\\
& \leq \sum_{i=1}^{N}|a_i|^2 d(z_i) \int_{S_i}|\nabla w_i(z) |^2 dA(z) + 
\sum_{i=1}^{N}|a_i|^2 d(z_i) \int_{S_i} |w_i(z)|^2 dA(z) \\
&\leq C  \sum_{i=1}^{N} |a_i|^2 d(z_i) \norm w_i \norm _{H_1(\DD)}^2 + \sum_{i=1}^{N}|a_i|^2 d(z_i) \norm w_i \norm_\infty^2 |S_i|  \\
& \leq C \Big( \sum_{i=1}^{N} |a_i|^2+ \sum_{i=1}^{N}|a_i|^2 d(z_i) (1-|z_i|)^{\gamma} \Big) \\
& \leq C \sum_{i=1}^{\infty}|a_i|^2.
\end{align*}

Now we turn to the necessity of the conditions. Without loss of generality $ 0 \in \{z_i\} $ and $ \varepsilon = 1 $.
We should notice that also here there exists a weighted restriction operator, defined on the subspace of $H_1(\DD)$ of functions constant on hyperbolic discs $\Delta_1(z_i)$. Hence if a sequence is onto interpolating as in the Dirichlet space we can solve the interpolation problem with ``norm control", meaning that we can find $u\in H_1(\DD)$ such that 
\begin{align*}
     u|_{\Delta_1(z_i)}  & \equiv \sqrt{d(z_i)} a_i \\
     \norm u \norm_{H_1(\DD)}^2 &  \leq C  \norm a \norm_{\ell^2}^2.
\end{align*}
Considering such a $u_j$ which interpolates $\{\delta_{ij} \}_i$, then $u$ is also an admissible function for the condenser $\big( \Delta_1(z_i), \Delta_1(z_j) \big)$, for any $i\neq j.$ It is immediate that 
\[ \capacity_\DD\big( \Delta_1(z_j), \Delta_1(z_i) \big) \leq \norm u \norm_{H_1(\DD)}^2 \leq \frac{C}{d(z_j)} \]

Since also by conformal invariance of condenser capacity we have 
\begin{equation*}
\capacity_{\DD}\big( \Delta_1(z_i), \Delta_1(z_j) \big) \approx \frac{1}{\log\frac{1}{|I_{\phi_{z_i}(z_j)}|}} \approx \frac{1}{d_h(z_i,z_j)}
\end{equation*}the sequence is weakly separated. The capacitary condition then follows by the definition of condenser capacity and Lemma \ref{ComparisonCapacities}.

\end{proof}

\section{A quantitative version of Bishop's theorem}\label{BishopTheorem}

We are now in a position to prove Theorem \ref{Theorem B}. The only substantial difference with the proof of the non holomorphic case is that the function $v_i$ which in the proof of Theorem \ref{Theorem C} (that exist by our assumptions on the condenser capacity) they can no longer be used since they are not holomorphic. Their role will be played by the functions constructed in Lemma \ref{BishopFunction}.

\begin{proof}[Proof of Theorem \ref{Theorem B}] We will denote by $ C $ a general constant depending only on $ K_0 $. Suppose that $ \{z_i \} $ is $ K- $ strongly separated. Let us exclude initially a finite number of points from the sequence such that Lemmas \ref{SeparationLemma} and Propositions \ref{CondensersBlowUp}, \ref{NecessaryCondition} apply. We can always add these points in the sequence in the end.

In our construction we will need three types of building blocks. The functions constructed by Bishop, B\o e and the sequences of functions guaranteed by the strong separation hypothesis. First we perform a simple trick so that the sequence of functions coming from strong separation are uniformly bounded in modulus. We know that  there exist multipliers $ m_i \in \mathcal{M}(\mathcal{D}) $ such that $ \norm m_i \norm_{\mathcal{M}(\mathcal{D})} \leq K $ and $ m_i(z_j)= \delta_{ij} $. Consider the functions $ f_i:=\frac{m_i k_{z_i}}{d(z_i)} $. It is immediate that  $ \normD{f_i}^2 \leq K/d(z_i), f_i(z_j)=\delta_{ij} $ and $ \norm  f_i \norm _\infty \leq C $.

The sequence is also weakly separated by a constant $\varepsilon = \varepsilon(K)>0$.
Let $\gamma = 1-\varepsilon/4 $ so that Lemma \ref{SeparationLemma} applies. We know that $\cZ$ satisfies the condition in Proposition \ref{NecessaryCondition} for $\gamma$ as defined here, hence  by applying Proposition \ref{CondensersBlowUp}  (for $\beta = \gamma $ and $\alpha = 1 + \varepsilon / 4$ ) we get
\begin{equation*}
\capacity_\DD\Big( \Delta_1(z_i), \bigcup_{z_j\in \mathcal{V}_{\gamma}(z_i)}I_{z_j}^{\gamma} \Big) \leq \frac{C}{d(z_i)}.
\end{equation*}
Let now $ g_i $ be the functions that we get if we apply Lemma \ref{BishopFunction} to the condenser $ \Big( \Delta_1(z_i), \bigcup_{z_j\in \mathcal{V}_{\gamma}(z_i)}I_{z_j}^{\gamma} \Big) $. Finally, let  $ h_i $ be B\o e's function associated to $ z_i $ and $ S^{\gamma}(z_i). $ Multiply these functions together to get $ u_i:=f_ig_ih_i $.

Suppose we are given a sequence $ \alpha = \{a_i\} \in \ell^2(\NN) $. As in the non holomorphic case the obvious choice for the interpolating function would be $ F:= \sum_{i=1}^{\infty}a_i \sqrt{d(z_i)} u_i  $. At least formally $ F $ assumes the correct values, the problem now being that $ u_i $ do not have disjoint supports. Nevertheless $ u_i $ are small outside the regions $ S_i $, as defined in Lemma \ref{Regions}, in the following sense.
Assume $ z\not\in S_i $, then $ |u_i(z)| \leq C e^{-6d(z_i)} $, and if $ S(E_i):= \bigcup_{z_j\in\mathcal{V}_{\gamma}(z_i)}S^{\gamma}{(z_j)}  $ 
\begin{align*}
 \int_{S(E_i)} |u_i'|^2dA & \leq 3 \int_{S(E_i)} |f_i'g_ih_i|^2dA \\  & \,\,\,\,\,\,\,\,\,\,\,\,\,\,\,\,\,  +3\int_{S(E_i)}  |f_ig_i'h_i|^2dA+3\int_{S(E_i)} |f_ig_ih_i'|^2dA \\
& \leq Ce^{-6d(z_i)}\int_{S(E_i)} |f_i'|^2dA \\
 & \,\,\,\,\,\,\,\,\,\,\,\,\,\,\,\,\, +C\int_{S(E_i)} |g_i'|^2dA +  Ce^{-6d(z_i)}\int_{S(E_i)} |h_i'|^2dA \\
& \leq C e^{-6d(z_i)}.
\end{align*} 
In the same way \[ \int_{\DD\setminus S^{\gamma}(z_i)}|u_i'|^2dA \leq C e^{-6d(z_i)}.\]

 Set $ F_N=\sum_{i=1}^{N}a_i \sqrt{d(z_i)} u_i $. And estimate as follows
\begin{align*}
\int_{\DD} |F_N'|^2dA & \leq 2\sum_{i=1}^{N}|a_i|^2 d(z_i) \int_{S_i} |u_i'|^2dA + 2\Big( \sum_{i=1}^{N}|a_i| \sqrt{d(z_i)} \Big[ \int_{\DD\setminus S_i}|u_i'|^2dA \Big]^{\frac{1}{2}} \Big)^2 \\
& \leq 2\sum_{i=1}^{N}|a_i|^2 d(z_i) \int_{\DD} |u_i'|^2dA + c\Big( \sum_{i=1}^{N}|a_i| \sqrt{d(z_i)} e^{-6d(z_i)} \Big)^2 \\
& \leq C\sum_{i=1}^{N}|a_i|^2  + C \sum_{i=1}^{N}|a_i|^2 \sum_{i=1}^{N} d(z_i)e^{-12d(z_i)} \\
& \leq C \sum_{i=1}^{\infty}|a_i|^2.
\end{align*}
The constant $ C $ above is independent of $ N $ because every weakly separated sequence satisfies  $ \sum_{i=1}^{\infty}d(z_i)e^{-12d(z_i)}<+\infty $ (see for example \cite{Arcozzi16}). Hence, $ \normD{F_N} \leq C \sum_{i=1}^{\infty}|a_i|^2 $. By choosing a $ weak-* $ cluster point of the sequence we have a function $ f $ which solves the interpolation problem. Since $ C $ depends only on $ K_0 $ the interpolation constant can be chosen uniformly, as in the statement.
\end{proof}

\section{Onto interpolation in \texorpdfstring{$ \mathcal{D} $}{D} for finite measure sequences}\label{DirichletInterpolation}

The necessity of the capacitary condition comes from Proposition \ref{CapacitarySufficient} together with Lemma \ref{LogarithmicCondensers}. The other direction follows from the next proposition.

\begin{prop}
If $ \{z_i\} $ has finite associated measure, then,
\begin{equation*}
	(WS)+(CC) \implies (OI).
\end{equation*}
\end{prop}

\begin{proof} Since this proof requires an inductive argument on finite subsets of the sequence we will be more careful with our constants.

Let $ \{z_i\} $ be a sequence as in the statement.  We can choose a constant $ C_0>1 $, which depends only on the sequence, that is large enough such that for all $ z_i\in \{z_i\} $ there exists $ f_i\in\mathcal{D} $ with $ \normD{f_i}^2\leq C_0 / d(z_i) $, $ f_i(z_i)=1 $ and $ |f_i(z_j)|^2 \leq C_0 e^{-d(z_i)} $ if $ j \neq i $. The functions $ f_i $ are constructed by multiplying B\o e's function $ f_{z_i,\gamma} $ with the function $ g_{z_i} $ from Lemma \ref{BishopFunction}. Also, by the quantitative version of Bishop's Theorem, there exists a constant $ C_1>1 $ such that for every  sequence of points $ \mathcal{E} \subseteq \DD $ which is $ K- $  strongly separated, $ K\leq 4C_0+1 $, satisfies $ \text{OntoInt}(\mathcal{E}) \leq C_1  $. By deleting a finite number of points in the sequence we can ensure that
\begin{equation*}
e^{-d(z_i)} \leq \frac{(1-1/\sqrt{2})^2}{4C_1C_0^2}\frac{1}{d(z_i)}, \,\, \text{for all}\,\,\, i\in\NN,
\end{equation*} 
and
\begin{equation*}
\sum_{j=1}^{\infty}\frac{1}{d(z_j)} \leq 1.
\end{equation*} 

For some $ N\in\NN $ set 
\begin{equation}
M_N := \sup_{\substack{\mathcal{E}\subseteq \{z_i\}\\ |\mathcal{E}|=N}} \text{StrongSep}(\mathcal{E}), \,\,\,\,  A_N:= \sup_{\substack{\mathcal{E}\subseteq \{z_i\}\\ |\mathcal{E}|=N}} \text{OntoInt}(\mathcal{E}).
\end{equation}

Notice that $ M_1\leq 1 $, but \textit{a priori} we don't even know if $ M_N<\infty $ for $ N>1 $. We claim that  $ M_N \leq 4C_0+1  $, for all $ N \geq 1 $. Suppose the statement is true for some $ N\geq 1  $. Consider any $ \mathcal{E} \subseteq \{z_i\} $ such that $ |\mathcal{E}|=N+1 $ and fix some $ z_i\in\mathcal{E} $. Let $ f_i $ be the function as defined above. By the induction hypothesis we have $ A_N\leq C_1 $, hence, there exists $ g_i\in\mathcal{D} $ such that $ g_i(w)+f_i(w)=0,  $ for all $ w\in\mathcal{E}\setminus\{z_i\} $ with 
\begin{equation*}
\normD{g_i}^2 \leq C_1 \sum_{w\in\mathcal{E}\setminus \{z_i\}}\frac{|f_i(w)|^2}{d(w)} \leq C_0 C_1 e^{-d(z_i)} \sum_{j=1}^{\infty}\frac{1}{d(z_j)} \leq \frac{1}{4d(z_i)}.
\end{equation*}

Furthermore,  
\begin{align*}
|g_i(z_i)|^2 & \leq \normD{g_i}^2 \normD{k_{z_i}}^2 \\
 &\leq C_0C_1 e^{-d(z_i)}C_0d(z_i)\\ 
  &\leq C_1C_0^2 d(z_i)e^{-d(z_i)} \\
  &\leq (1-1/\sqrt{2})^2.
\end{align*}

Finally, consider the function $ h_i:=(f_i+g_i)/(f_i(z_i)+g_i(z_i)). $ By definition $ h_i(z_i)=1 $ and $ h_i(w)=0 $ for all $ w\in\mathcal{E}\setminus\{z_i\} $. Also,
\begin{equation*}
\normD{h_i}^2 = \frac{\normD{f_i+g_i}^2}{|f_i(z_i)+g_i(z_i)|^2} \leq 4 \normD{f_i}^2+ 4 \normD{g_i}^2 \leq \frac{4C_0}{d(z_i)}+\frac{1}{d(z_i)}.
\end{equation*}Since $ z_i $ was arbitrary by definition of $ M_{N+1} $, we have that $ M_{N+1}\leq 4C_0+1 $.
 The induction is complete and it gives that $ \limsup_{N\to\infty}M_N \leq 4C_0+1 < \infty $. Therefore $ \{ z_i \} $ is strongly separated.

\end{proof}


\section{Some remarks on the capacitary condition} \label{Unions}

\subsection{A stronger condition implying the capacitary condition}\label{CapacitarySufficient}

\begin{proof}[Proof of Theorem \ref{THEOREM D}] The idea of the proof is an old one, originally due to Shapiro and Shields, and it amounts to use Nevanlinna-Pick property in order to compensate for the lack of Blaschke products in the Dirichlet space. Let $ z_i $ a point in the sequence. Let $S^\gamma_\circ(z_i)$ the set of points $z_j$ in the sequence such that $z_i$ has $\gamma$-uninterrupted view of $z_j$.  For each $ z_j\in S^\gamma_\circ(z_i) $ consider the multiplier $ \psi_{ij}\in\mathcal{M}_1(\mathcal{D}):= \{m\in\mathcal{M}(\mathcal
	D): \norm m \norm_{\mathcal{M}(\mathcal{D})} \leq 1 \} $ which vanishes at $ z_j $ and maximizes $ \Re\psi_{ij}(z_i) $. Due to the Nevanlinna Pick property of $  \mathcal{D} $ (see \cite[Theorem 9.43]{AglerMcCarthy00}, 
	
	\[  \psi_{ij}(z_i) = 1- \frac{|\inner{k_{z_i}}{k_{z_j}}|^2}{\normD{k_{z_i}^2}\normD{k_{z_i}}^2}. \] 
	
	Then consider a $ weak-* $ cluster point of the sequence  $ \psi_{ij_1}^2\psi_{ij_2}^2 \dots \psi_{ij_N}^2 $, where $ S^\gamma_\circ(z_i) = \{z_{j_1},z_{j_2},\dots \}  $. Let's call this cluster point $ \psi_i $. Obviously it vanishes on all points in $ S^\gamma_\circ(z_i)  $ and at $ z_i $ takes the value 
	\begin{equation*}
		\psi_i(z_i)= \prod_{z_j\in S^\gamma_\circ(z_i)}1- \frac{|\inner{k_{z_i}}{k_{z_j}}|^2}{\normD{k_{z_i}^2}\normD{k_{z_i}}^2}.
	\end{equation*} 
	To estimate this infinite product we use our hypothesis
	\begin{align*}
	 \sum_{z_j\in S^\gamma_\circ(z_i) } \frac{|\inner{k_{z_i}}{k_{z_j}}|^2}{\normD{k_{z_i}^2}\normD{k_{z_i}}^2} &
			 \lesssim \sum_{z_j\in S^\gamma_\circ(z_i) } \frac{\Big( \log\frac{1}{|1-\overline{z_i}z_j|}\Big)^2}{d(z_i)d(z_j)}	\\
			& \lesssim \sum_{z_j\in S^\gamma_\circ(z_i) }\frac{d(z_i)}{d(z_j)} \\
			& \leq C.	
	\end{align*} Considering also that the sequence is weakly separated, hence every individual term of the infinite product is bounded away from zero in modulus we can conclude that  $ |\psi_i(z_i)| $ is bounded below by a constant independent of $ i $. Consider the functions $ f_i:=k_{z_i} \psi_i/(\normD{k_{z_i}}^2|\psi_i(z_i)|) $. The same argument as in the proof of Lemma \ref{NecessaryCondition} applied to the functions $ f_i $, shows that 
	\begin{equation*}
		\capacity_\DD\Big( \Delta_1(z_i), \bigcup_{z_j \in  S^\gamma_\circ(z_i)} I_{z_j} \Big) \leq \frac{C}{d(z_i)}.
	\end{equation*} But then the estimate for the capacitary condition is immediate by Proposition \ref{CondensersBlowUp} and Lemma \ref{SeparationLemma}.
	\begin{align*}
		\capacity_\DD\Big( \Delta_1(z_i), \bigcup_{z_j \in  S^\gamma(z_i)} I_{z_j} \Big)  
	&	\leq \capacity_\DD\Big( \Delta_1(z_i), \bigcup_{z_j \in  S^\gamma_\circ(z_i)} I_{z_j}^{\gamma} \Big) \\ &\lesssim \capacity_\DD\Big( \Delta_1(z_i), \bigcup_{z_j \in  S^\gamma_\circ(z_i)} I_{z_j} \Big) \\
		& \leq  C/d(z_i).
	\end{align*}
\end{proof}

\subsection{A negative result} \label{Counterexample} In order to construct the counter example in Theorem \ref{Thoerem F} we will exploit the standard Bergman tree in the unit disc and the relevant analysis of onto interpolating sequences in the Dirichlet space of the Bergman tree carried by Arcozzi Rochberg and Sawyer in \cite{Arcozzi16}. 

For $ n\in\NN $ and $ k=1,2,\dots 2^n $, let $ z(k,n)=(1-2^{-n})e^{2i\pi\frac{k}{2^n}} $ and denote by $ \tau $ the collection of all such points. For $\alpha = z(k,n)\in\tau $ we will refer to $ n $ as the level of $ \alpha $ and denote it by $ d_\tau(a) $. The set $ \tau $ can be given a structure of a rooted tree by declaring the origin to be the root of the tree and if $ \alpha,\beta\in\tau $, we say that a directed edge connects $ \alpha $ to $ \beta $ if $ d_\tau(\alpha)+1=\level{\beta} $ and $ S(\alpha) \supseteq S(\beta) $. In this case we say that $ \beta $ is a child of $ \alpha $. Each $ \alpha $ has two children $ \sigma_+\alpha, \sigma_-\alpha $ and a unique predecessor $ \alpha^{-1} $. The tree is partially order by the relation $ \alpha \leq \beta  $ iff $ S(\alpha)\subseteq S(\beta) $.

The tree model is convenient because it gives a necessary condition for interpolation in the Dirichlet space in terms of capacities defined on trees, which are highly more computable  with respect to their continuous counterparts. In fact, in \cite{Arcozzi16}, Arcozzi Rochberg and Sawyer gave the following necessary  condition for onto interpolation.

\begin{prop}[Tree Capacitary Condition] Suppose that $ \{ z_n \} \subseteq \tau $ is an onto interpolating sequence for the Dirichlet space. Then
 \begin{equation*}
\inf \Big\{ \sum_{\beta\in\tau \setminus \{0 \} }|\nabla f(\beta)|^2: f:\tau \to \mathbb{C}, f(\alpha)=1, f(\gamma)=0 \,\,\, , \forall \gamma\in \{ z_i \} \setminus \{\alpha \} \Big\} \leq \frac{C}{d(\alpha)}
\end{equation*} for all $ \alpha \in \{z_i\} $. Where $ \nabla f(\beta) := f(\beta)-f(\beta^{-1}) $.We will denote the quantity on the left by $ \capacity_\tau(\alpha; \{z_i\}\setminus \{\alpha \} ) $.
\end{prop}

\begin{lem} $ d_\tau(z)+1\approx d(z) $, $ z\in\tau $.
\end{lem}
\begin{proof}
Let $ z=z(n,k) $, then $ |z|=1-2^{-n} $. Without loss of generality $ n\neq 0 $. Therefore $ d_h(0,z)=\log_2(2^{n+1}-1) \leq n+1 $. On the other hand $ \log_2(2^{n+1}-1) \geq \log_2(2^n+1) \geq n $.
\end{proof}

The tree capacitary condition is much easier to analyse, mainly because there exists a recursive formula for its computation \cite[p. 32]{Arcozzi16}.

Given $ \alpha, \beta \in \tau $, $ \alpha<\beta $ and $ U_{\pm} \subseteq S(\sigma_{\pm}\beta) \cap \tau $ we have
\begin{equation*}\label{Recursive}
\capacity_\tau(\alpha,U_+\cup U_-)= \frac{\capacity_\tau(\alpha, U_+)+\capacity_\tau(\alpha, U_-)}{1+d_\tau(\alpha,\beta)[\capacity_\tau(\alpha, U_+)+\capacity_\tau(\alpha, U_-)]},
\end{equation*}
where $d_\tau(\alpha,\beta)$ is the {\it graph distance } of the points $\alpha, \beta.$

\begin{proof}[Proof of Theorem \ref{Thoerem F}] Let $ z_0\in\tau $, and set $ N=\level{z} $. Assume for simplicity that $ \sqrt{N}  $ is an integer. Consider also the points $ \{w_i\}_{i=0}^{\sqrt{N}} $, where $ w_0=z $ and $ w_{i+1}=\sigma_+w_i $, and the points $ z_i=\sigma_-^{(N)}w_i, 0\leq i\leq \sqrt{N} $. Due to the shape of the representation of this configuration of points as a graph we shall write $ comb(z_0):=\{z_1,\dots,z_{\sqrt{N}} \} $.

Let us start with an estimate of $ \capacity_\tau(z_0,comb(z_0)) $. This can be done by applying the recursive formula \eqref{Recursive}. Let $ c_i=\capacity_\tau(w_i,\{z_{i+1},\dots,z_{\sqrt{N}}\}) , \,\, i<\sqrt{N} $ and $ c_{\sqrt{N}}=0 $. Then the recursive formula gives 
\begin{equation*}
c_{i-1}=\frac{\frac1N+c_i}{1+\frac1N+c_i}=\rho  \begin{pmatrix}
1 & \frac1N \\
1 & \frac1N + 1 \\
\end{pmatrix} (c_i).
\end{equation*} Where $ \rho $ is the map  $ \rho : M_2(\mathbb{C}) \mapsto \text{M\"ob}, \begin{pmatrix}
a & b \\
c & d \\
\end{pmatrix} \overset{\rho}{\mapsto} \Big(z \mapsto \frac{az+b}{cz+d} \Big) $. Since $ \rho $ is a homomorphism we conclude that \begin{equation*}
\capacity_\tau(z_0,comb(z_0))=c_0=\rho \Big( \begin{pmatrix}
1 & \frac1N \\
1 & \frac1N + 1 \\
\end{pmatrix}^{\sqrt{N}} \Big) (0).
\end{equation*} After diagonalizing the matrix we get 
\begin{equation*}
\begin{pmatrix}
1 & \frac1N \\
1 & \frac1N + 1 \\
\end{pmatrix}^{\sqrt{N}} = \begin{pmatrix}
\delta_2 & \Delta \\
1 & 1 \\
\end{pmatrix} \begin{pmatrix}
(1-\Delta)^{\sqrt{N}} & 0 \\
0 & (1-\delta_2)^{\sqrt{N}} 
\end{pmatrix} \begin{pmatrix}
\delta_2 & \Delta \\
1 & 1 \\
\end{pmatrix}^{-1},
\end{equation*} where \begin{equation*}
\Delta=\frac{-\frac1N+\sqrt{(\frac1N)^2+4\frac1N}}{2} \,\,\,\, , \delta_2=\frac{-\frac1N-\sqrt{(\frac1N)^2+4\frac1N}}{2}.
\end{equation*} A simple algebraic manipulation of the previous expression leads to 
\begin{equation*}
\frac{c_0}{1/\sqrt{N}}=\frac{c_0}{\Delta}\sqrt{N}\Delta=\frac{1-\big(\frac{1-\Delta}{1-\delta_2}\big)^{\sqrt{N}}}{1-\frac{\Delta}{\delta_2}\big(\frac{1-\Delta}{1-\delta_2}\big)^{\sqrt{N}} } \sqrt{N}\Delta \overset{N \to \infty}{\longrightarrow} \frac{e^2-1}{e^2+1} > 0.
\end{equation*}Because $ \Delta/\delta_2 \to -1 $, $ \big(\frac{1-\Delta}{1-\delta_2}\big)^{\sqrt{N}} \to e^{-2} $. Hence, for $ N $ sufficiently big $ c_0 \geq \frac{1}{10\sqrt{N}} $.

Also we can calculate the total mass that each $ comb(z_0) $ carries

\begin{equation*}
\sum_{i=1}^{\sqrt{N}}\frac{1}{d(z_i)} \lesssim \sum_{i=1}^{\sqrt{N}}\frac{1}{\level{z_i}} 
\lesssim \sum_{i=1}^{\sqrt{N}}\frac{1}{N} 
=\frac{1}{\sqrt{N}} \lesssim \frac{1}{\sqrt{d(z_0)}}.
\end{equation*}

A last remark is the following. There exists an $ \eta<1 $, which can be chosen independently of $ N $ such that if $ 1 \leq i\neq j \leq \sqrt{N} $ then $ S^\eta(z_i)\cap S^\eta(z_j) = \emptyset $.

Consider now a new sequence of points $ \{\omega_i\} $ such that  for any $ \alpha\in comb(\omega_i),  \beta\in comb(\omega_j) $  $ S^\eta(\alpha) \cap S^\eta(\beta)=\emptyset  $ for $ i\neq j $ and some $0 < \eta < 1 $, and also $ \sum_{i=1}^{\infty}1/\sqrt{d(\omega_i)} < \infty$. By a theorem of Axler \cite{Axler92} $ \{\omega_i\} $ has a universally interpolating subsequence. We can assume without loss of generality that $ \{\omega_i\} $ itself is universally interpolating. Set $ \{w_i\} = \bigcup_{i=1}^\infty comb(\omega_i) $. It is clear that the union of the two sequences it cannot be onto interpolating because it fails the tree capacitary condition 
\begin{equation*}
\capacity_\tau(z_j,\{z_i\}_{j\neq i}\cup\{w_i\}) \geq \capacity_\tau(z_j,comb(z_j)) \geq \frac{1}{10\sqrt{\level{z_j}}} .
\end{equation*}

Nevertheless $ \{w_i\} $ is onto interpolating by Theorem \ref{Theorem A} because there exists $ \eta<1 $ such that  $ S^\eta(w_i)\cap S^\eta(w_j) = \emptyset $, it has finite associated measure and it is weakly separated.
\end{proof}

\subsection*{Concluding remarks}

If we have a sequence $ \{z_i\} \subseteq \tau $, it would be interesting to know whether the tree capacitary condition implies the capacitary condition, for that would mean that the onto interpolating sequences for the tree coincide with the onto interpolating sequences for the Dirichlet space, at least for finite measure sequences, which are much easier to understand mainly due to the recursive relations for tree capacities. 

Another question which remains open is if the characterization of onto interpolating sequences carries over to the case of infinite associated measure, something that is suggested by the analogous result for $ H_1(\DD) $. In fact if one examines the proof of Theorem \ref{Theorem A} can see that that would be true if $ \ell^\infty(\NN) \subseteq \{ \{ f(z_i) \} : f\in\mathcal{D} \} $ for every onto interpolating sequence.
There are even some questions in $H_1(\DD)-$interpolation which remain open. For example in our definition we introduced a parameter $\varepsilon$ and it is not clear at all how the interpolation constant depends on the parameter $\varepsilon$. Moreover our definition of $H_1(\DD)$ interpolation, although fit for our purposes, is only one of the natural definition that one could come up with for example, instead one could ask only that the interpolating function $u\in H_1(\DD)$ is only on average equal to the data, i.e. 
\[ \frac{1}{|\Delta_\varepsilon(z_i)|}\int_{\Delta_\varepsilon(z_i)} u dA = \sqrt{d(z_i)} a_i. \]
Such questions have not been investigated, but it is the authors opinion that it would be very interesting to explore them further.

\subsection*{Acknowledgements}
I am indebted to my supervisor Nicola Arcozzi for introducing me to the problems considered in this article and for his help with the construction of the counterexample in Section \ref{Counterexample}. Also I would like to thank Christopher Bishop for his courtesy we have included a detailed exposition of his proof in Section \ref{BishopTheorem}. Furthermore, I would like to thank Pavel Mozolyako, Dimitris Betsakos and Filippo Sarti for interesting discussions on harmonic measure. I thank also the anonymous referees for their careful comments that have greatly improved the clarity of the exposition.

\bibliographystyle{plain}
\bibliography{Bibliography}

\begin{thebibliography}{10}

\bibitem{Adams}
D.~R. Adams and L.~I. Hedberg.
\newblock {\em Function spaces and potential theory}.
\newblock Springer-Verlag Berlin Heidelberg, 1996.

\bibitem{AglerMcCarthy00}
J.~Agler and J.E. McCarthy.
\newblock {\em Pick interpolation and Hilbert function spaces}.
\newblock Graduate studies in mathematics. Amer. Math. Soc., 2002.

\bibitem{Aleman17}
A.~Aleman, M.~Hartz, J.~E. McCarthy, and S.~Richter.
\newblock {Interpolating sequences in spaces with the complete {P}ick
  property}.
\newblock {\em Int. Math. Res. Not. IMRS}, 2019(12):3832--3854, 10 2017.

\bibitem{Arcozzi16}
N.~Arcozzi, R.~Rochberg, and E.~Sawyer.
\newblock {Onto interpolating sequences for the Dirichlet space}.
\newblock {\em ArXiv e-prints}, 2016arXiv160502730A, 2016.

\bibitem{Arcozzi2020}
N.~Arcozzi, R.~Rochberg, E.T. Saywer, and .~B.D. Wick.
\newblock {\em The Dirichlet space and related function spaces}.
\newblock Amer. Math. Soc., 2020.

\bibitem{Arcozzi10}
N.~Arcozzi, R.~Rochbergand~E. Sawyer, and B.~Wick.
\newblock Bilinear forms on the {D}irichlet space.
\newblock {\em Anal. PDE}, 3(1):21--47, 2010.

\bibitem{Axler92}
S.~Axler.
\newblock Interpolation by multipliers of the {D}irichlet space.
\newblock {\em Quart. J. Math. Oxford Ser. (2)}, 43(172):409--419, 1992.

\bibitem{Bishop94}
C.~J. Bishop.
\newblock Interpolating sequences for the {D}irichlet space and its
  multipliers.
\newblock {\em Preprint,
  http://www.math.stonybrook.edu/~bishop/papers/mult.pdf}, 1994.

\bibitem{Boe2005}
B.~B{\o}e.
\newblock An interpolation theorem for {H}ilbert spaces with
  {N}evanlinna-{P}ick kernel.
\newblock {\em Proceedings of the American Mathematical Society},
  133(07):2077--2081, January 2005.

\bibitem{Carleson1952}
L.~Carleson.
\newblock On the zeros of functions with bounded {D}irichlet integrals.
\newblock {\em Math. Z.}, 56(3):289--295, September 1952.

\bibitem{Carleson1958}
L.~Carleson.
\newblock An interpolation problem for bounded analytic functions.
\newblock {\em Amer. J. Math.}, 80(4):921, oct 1958.

\bibitem{Carleson1962}
L.~Carleson.
\newblock Interpolations by bounded analytic functions and the corona problem.
\newblock {\em Ann. of Math.}, 76(3):547--559, 1962.

\bibitem{Cascante2012}
C.~Cascante and J.~M. Ortega.
\newblock On a characterization of bilinear forms on the {D}irichlet space.
\newblock {\em Proc. Amer. Math. Soc.}, 140(7):2429--2440, 2012.

\bibitem{Dubinin14}
V.~N. Dubinin and N.~G. Kruzhilin.
\newblock {\em Condenser capacities and symmetrization in geometric function
  theory}.
\newblock Springer Basel, 2014.

\bibitem{Boe}
B.~B\o e.
\newblock Interpolating sequences for {B}esov spaces.
\newblock {\em J. Funct. Anal.}, 192(2):319 -- 341, 2002.

\bibitem{Mashreghi14}
O.~El-Fallah, K.~Kellay, J.~Mashreghi, and T.~Ransford.
\newblock {\em A primer on the Dirichlet space}.
\newblock Cambridge Tracts in Mathematics. Cambridge University Press, 2014.

\bibitem{MarshallGarnett}
J.~B. Garnett and D.~E. Marshall.
\newblock {\em Harmonic measure}.
\newblock New Mathematical Monographs. Cambridge University Press, 2005.

\bibitem{Kellay2011}
K.~Kellay and J.~Mashreghi.
\newblock On zero sets in the {D}irichlet space.
\newblock {\em J. Geom. Anal.}, 22(4):1055--1070, April 2011.

\bibitem{Marshall94}
D.~E. Marshall and C.~Sundberg.
\newblock Interpolating sequences for the multipliers of the {D}irichlet space.
\newblock {\em Preprint,
  https://sites.math.washington.edu/~marshall/preprints/interp.pdf}, 1994.

\bibitem{Mashreghi2010}
J.~Mashreghi, T.~Ransford, and M.~Shabankhah.
\newblock Arguments of zero sets in the {D}irichlet space.
\newblock In {\em {CRM} Proceedings and Lecture Notes}, pages 143--148.
  American Mathematical Society, April 2010.

\bibitem{McCarthy2003}
J.~E. McCarthy.
\newblock Pick's theorem—what's the big deal?
\newblock {\em Amer. Math. Monthly}, 110(1):36--45, 2003.

\bibitem{Nagel1982}
A.~Nagel, W.~Rudin, and J.~H. Shapiro.
\newblock Tangential boundary behavior of functions in {D}irichlet-type spaces.
\newblock {\em Ann. of Math.}, 116(2):331--360, 1982.

\bibitem{Nevanlinna1919}
R.~Nevanlinna.
\newblock {\"U}ber beschr\"ankte funktionen die in gegebenen punkten
  vorgeschriebene werte annehmen.
\newblock {\em Ann. Acad. Sci. Fenn. Ser. A}, 13:1--71, 1919.

\bibitem{Pick1915}
G.~Pick.
\newblock {\"U}ber die beschr\"ankungen analytischer funktionen, welche durch
  vorgegebene funktionswerte bewirkt werden.
\newblock {\em Math. Ann.}, 77(1):7--23, March 1915.

\bibitem{Richter2004}
S.~Richter, W.~T. Ross, and C.~Sundberg.
\newblock Zeros of functions with finite {D}irichlet integral.
\newblock {\em Proc. Amer. Math. Soc.}, 132(8):2361--2365, 2004.

\bibitem{Schark1961}
I.~Schark.
\newblock Maximal ideals in an algebra of bounded analytic functions.
\newblock {\em Indiana Univ. Math. J.}, 10:735--746, 1961.

\bibitem{Shapiro61}
H.~S. Shapiro and A.~L. Shields.
\newblock On some interpolation problems for analytic functions.
\newblock {\em Amer. J. Math.}, 83(3):513--532, 1961.

\bibitem{Shapiro1962}
H.~S. Shapiro and A.~L. Shields.
\newblock On the zeros of functions with finite {D}irichlet integral and some
  related function spaces.
\newblock {\em Math. Z.}, 80(1):217--229, December 1962.

\bibitem{Stegenga80}
D.~A. Stegenga.
\newblock Multipliers of the {D}irichlet space.
\newblock {\em Illinois J. Math.}, 24(1):113--139, 1980.

\end{thebibliography}

\end{document}